\documentclass{amsart}

\usepackage{verbatim}
\usepackage{color}

\renewcommand{\a}{\alpha}

\newcommand{\e}{\epsilon}

\newtheorem{thm}{Theorem}
\newtheorem{prop}[thm]{Proposition}
\newtheorem{lem}[thm]{Lemma}

\newtheorem{rmk}[thm]{Remark}

\numberwithin{equation}{section}

\title[Local Existence]{A Local Existence Result for Poincar\'e-Einstein metrics}

\author{Matthew J. Gursky}
\address{Department of Mathematics
         University of Notre Dame\\
         Notre Dame, IN 46556}

\author{G\'{a}bor Sz\'{e}kelyhidi}
\address{Department of Mathematics\\
         University of Notre Dame\\
         Notre Dame, IN 46556}

\begin{document}

\begin{abstract} Given a closed Riemannian manifold $(M, g_M)$ of dimension $n \geq 3$, we prove the existence of a conformally compact Einstein metric $g_{+}$ defined on a collar neighborhood $M \times (0,1]$ whose  conformal infinity is $[g_M]$.

\end{abstract}

\maketitle

\section{introduction}

Let $X$ be the interior of a compact manifold with boundary $\overline{X}$ of dimension $n+1$, and let $M  = \partial X$ denote the boundary.  A metric $g_{+}$ defined on $X$ is said to be {\em conformally compact} if there is a defining function $\rho \in C^{\infty}(X)$ with $\rho > 0$ and $d\rho \neq 0$ on $\partial X$, such that $\rho^2 g$ extends to a metric
$\overline{g}$ on $\overline{X}$.  Since we can multiply $\rho$ by any smooth positive function on $\overline{X}$, a conformally compact metric naturally defines a conformal class of metrics $[ \overline{g} ]$ on $M = \partial X$, called the {\em conformal infinity} of $(X,g)$.

If in addition $g_{+}$ satisfies the Einstein condition, which we normalize by
\begin{align} \label{EE}
Ric(g_{+}) = -n g_{+},
\end{align}
then we say that $(X,g_{+})$ is a {\em Poincar\'e-Einstein} (P-E) manifold.  The motivating example of P-E manifolds is the Poincar\'e ball model of hyperbolic space $(\mathbf{B}^{n+1}, g_{\mathbf{H}})$,  and in this case the conformal infinity is the conformal class of the round sphere $\mathbf{S}^n = \partial \mathbf{B}^{n+1}$.
P-E manifolds play a fundamental role in the Fefferman-Graham theory of conformal invariants (see \cite{FG}), and in the AdS/CFT correspondence in quantum field theory (see, for example, \cite{Maldecena}).  Our main interest in this paper is the question of existence:  given a conformal class $[ g_M ]$ on the closed manifold $M = \partial X$, is there a Poincar\'e-Einstein metric $g_{+}$ defined in $X$ whose conformal infinity is $[g_M]$?

A seminal existence result was proved by Graham-Lee in \cite{GL}: given a metric $\gamma$ sufficiently close to the round metric $\gamma_0$ on the sphere $\mathbf{S}^n$, there is a Poincar\'e-Einstein metric $g_{+}$ on the ball $\mathbf{B}^{n+1}$ whose conformal infinity is $[\gamma]$.  Later, Lee \cite{LeeMemoirs} extended this prove the existence of P-E metrics whose conformal infinity is sufficiently close to the conformal infinity of a given P-E metric, provided the linearized operator (suitably defined) is invertible.  Anderson \cite{An1} proved a more general existence result on $\mathbf{S}^3$: any conformal class with positive Yamabe invariant is the conformal infinity of a P-E metric.

By contrast, in joint work with Q. Han (\cite{GH}) the first author proved a non-existence result for conformal classes on $\mathbf{S}^7$:  there are infinitely many conformal classes (which can be taken in different components of the space of PSC metrics) which {\em cannot} be the conformal infinity of a P-E metric in the ball $\mathbf{B}^8$.  The proof uses in a crucial way the work of Gromov-Lawson \cite{GrL} on the space of PSC metrics on $\mathbf{S}^7$, demonstrating that the existence is Poincar\'e-Einstein fillings is influenced by the topology of $X$ as well as the geometry of the conformal infinity.

 Since there are obstructions to the global existence of Poincar\'e-Einstein fillings, in this paper we consider a local version: given a closed Riemannian manifold $(M, g_M)$, we find a conformally compact Einstein metric $g_{+}$ defined on a collar neighborhood $M \times (0,1]$ such that the conformal infinity of $g_{+}$ is $[g_M]$ (a more precise statement is given below).  If $M$ is real analytic, then there is always a P-E metric defined on a collar neighborhood $M \times (0,1]$; this was proved when $M$ is odd-dimensional by Fefferman-Graham in \cite{FG}, and in the even-dimensional case by Kichenassamy in \cite{Kiss}.  Also, LeBrun used twistor methods to construct an ASD Poincar\'e-Einstein metric in a collar neighborhood of any real analytic three-manifold, see \cite{LeBrun}.    Our interest in this paper is therefore in the $C^{\infty}$ category, and our main result is: \\

\begin{thm}\label{thm:mainthm}  Let $(M,g_M)$ be a smooth, connected, closed manifold of dimension $n \geq 3$.  Then there is a metric $g_{+}$ defined on $X = M \times (0,1]$ with the following properties: \\

\noindent $(i)$ $(X,g_{+})$ is a manifold with boundary $\partial X = M \times \{1\} \cong M$ satisfying the Einstein condition:
\begin{align*}
Ric(g_{+}) + n g_{+} = 0.
\end{align*}

\noindent $(ii)$ $(X,g_{+})$ is conformally compact with conformal infinity given by $(M,[g_M]).$  More precisely, there is a defining function $\rho \in C^{\infty}(X)$ such that $\bar{g} = \rho^2 g_{+}$
defines a $C^0$-metric on the compact manifold with boundary $\overline{X} = M \times [0,1 ]$ with
\begin{align*}
\bar{g}\big \vert_{M \times \{ 0 \}} = g_M.
\end{align*}
\end{thm}

To give a sketch of our approach we begin by considering the model case.  Let $dx^2$ denote the Euclidean metric on $\mathbf{R}^n$, and on $\mathbf{H}^{n+1} = \mathbf{R} \times \mathbf{R}^n$
let $g_{\mathbf{H}}$ denote the hyperbolic metric
\begin{align*}
g_{\mathbf{H}} = dt^2 + e^{2t} dx^2.
\end{align*}
We can recover the standard upper half-space model by letting $t = \log \frac{1}{y}$, so that
\begin{align*}
g_{\mathbf{H}}  = \dfrac{dy^2 + dx^2}{y^2}.
\end{align*}
In particular, restricting to $\{ (x,y)\ :\ x \in \mathbf{R}^n, y \in (0,1] \}$ we obtain an Einstein metric $g_{+} = g_{\mathbf{H}}$ on the manifold with boundary $\mathbf{H}_{+}^{n+1} = [ 0, \infty) \times \mathbf{R}^n$, whose compactification $\bar{g} = y^2 g_{\mathbf{H}}$ gives the Euclidean metric on the boundary.

Given a compact manifold $(M,g_M)$ and $\epsilon > 0$ small, as a first approximation we define the metric
\begin{align*}
g_\epsilon = dt^2 + e^{2t} \epsilon^{-2}g_M
\end{align*}
on $[0, \infty) \times M$.   On a fixed compact set, when $\epsilon >
0$ is small the metric $g_{\epsilon}$ is close to the hyperbolic
metric $g_{\mathbf{H}}$.  Our goal is to perturb $g_{\epsilon}$  to
obtain a Poincar\'e-Einstein metric $g_{+} = g_{\epsilon} + h$ on
$M_{+} = [0,\infty) \times M$.  If we compactify by letting $\bar{g} =
\epsilon^2 e^{-2t}g_{+}$, then assuming $h$ decays fast enough it
follows that $\bar{g}\big|_{y=0} = g_M$ as required.

One advantage of
rescaling $(M,g_M)$ and considering $g_{\e}$ is that the linearized
problem can be reduced, via a cutting and pasting method, to the
linearized problem on the model space $\mathbf{H}_{+}$, where Fourier
transform methods can be used. This is somewhat reminiscent of ``gluing''
problems along submanifolds in the literature, such as
Taubes~\cite{Taubes} and more specifically Brendle~\cite{Brendle} in the
context of gauge theory. One key difference in our setting is that our
model geometry is not a product.

It turns out that the metric $g_{\epsilon}$ is not a sufficiently good
approximation.  Roughly,
$g_{\epsilon}$ is a solution up to an error of order $\epsilon^2$, but
our estimates for the linearized operator require the error to be of
order smaller than $\epsilon^4$ in order to use a fixed point argument.   To remedy this we appeal to the
formal solutions of Fefferman-Graham \cite{FG} to `correct' $g_{\epsilon}$; see
Lemma \ref{lem:errorestimate} below.

As in the global existence problem for Einstein metrics we also need to
compensate for diffeomorphism invariance by introducing a
`gauge-fixed' version of the problem.  We will consider a slight
variant of the mapping defined by Graham-Lee in \cite{GL}, but the
essential idea is the same: we add a Lie derivative term {\em \`{a}
  la} DeTurck~\cite{DeTurck} in order to cancel out the degeneracies in the symbol of
the linearized operator.

To prove that a zero of the gauge-fixed mapping
is an Einstein metric, Graham-Lee used the Bianchi condition along with
a maximum principle argument (see Lemma 2.2 of \cite{GL}).   To prove the
analogous result in our setting we need to impose an appropriate boundary
condition on the `inner' boundary.  This introduces a number of technical
issues that have no obvious counterpart in the work of Graham-Lee or Graham.  For example,
we will see that our (gauge-fixed) linear operator will in general have a finite dimensional
cokernel, and we need to append the domain of the nonlinear mapping in order to
get surjectivity. In addition our boundary condition is not elliptic,
since it is underdetermined. One could attempt to add additional boundary
conditions such as those introduced by Schlenker~\cite{Sch} and Anderson~\cite{An08} to obtain
an elliptic boundary value problem, however it seems difficult to
identify the cokernels of these operators.  

In this context we should also mention the work of Chru\'{s}ciel-Delay-Lee-Skinner on boundary regularity for Poincar\'e-Einstien metrics~\cite{CDLS}, 
in which they construct a harmonic map on a collar neighborhood of the boundary using a perturbation argument (see Theorem 4.5).   However, 
they are imposing Dirichlet boundary conditions, and the invertibility of their linearized map follows from Theorem C of \cite{LeeMemoirs}. 

In the next section we will begin by introducing the nonlinear problem
and the `inner' boundary condition, and assuming the invertibility of the
linearized problem we prove our main result. The remainder of the paper
will be concerned with constructing a right inverse for the linearized
operator.

\subsection*{Acknowledgements}   The first author is supported in part by NSF grant DMS-1509633.  The second author is supported in part by
NSF grant DMS-1350696.

\ \\

\section{The nonlinear problem}  \label{SecNonlinear}

As in the Introduction, let $(M,g_M)$ be a compact $n$-manifold, and for $\epsilon > 0$ we define the metric $g_{\epsilon}$ on
$M\times [0,\infty)$ by
\[ g_\epsilon = dt^2 + e^{2t} \epsilon^{-2}g_M. \]
We want to find a symmetric $2$-tensor $h$ with sufficient decay at infinity so that
\begin{align} \label{goal}
\mathrm{Ric}(g_\epsilon + h) + n(g_\epsilon + h) = 0,
\end{align}
i.e., $g=g_{\epsilon} + h$ is a Poincar\'e-Einstein metric.
Before providing an outline of our argument, we begin with some
preliminary remarks and definitions.

We will work in weighted H\"older spaces $C^{k,\alpha}_\delta = e^{-\delta
  t}C^{k,\alpha}$, with the norm
\[ \Vert f \Vert_{C^{k,\a}_\delta} = \Vert e^{\delta t}
f\Vert_{C^{k,\a}} \]
in terms of the usual H\"older spaces (see Lee~\cite{LeeMemoirs} Chapter 3).  This norm extends to sections of the
various tensor bundles; e.g. $C_{\delta}^{k,\alpha}(S^2)$ will denote the space of symmetric two-tensors with respect to this norm.
We will choose the weight $\delta = 1$; in practice any weight $\delta\in(0,n)$ would
work, provided we start with a sufficiently good approximate solution.  Constructing a better approximate solution than $g_{\epsilon}$ is
the point of our first technical lemma:

\begin{lem} \label{lem:errorestimate} Given $(M,g_M)$, there are symmetric $2$-tensors $k^{(2)}, k^{(4)}$ defined on $M$ such that if
\begin{align} \label{gepp}
g_\epsilon' = g_\epsilon + k^{(2)} + e^{-2t}\epsilon^2 k^{(4)},
\end{align}
then
\begin{align} \label{app}
\Vert \mathrm{Ric}(g_\epsilon') +
  ng_\epsilon'\Vert_{C^{0,\alpha}_1} = O(\epsilon^6).
  \end{align}
\end{lem}
\begin{proof}  In \cite{FG}, Fefferman-Graham proved the existence of a one-parameter family of metrics $\gamma_r$ on $M$ such that the metric on $M \times (0, 1 ]$ given by
\begin{align*}
g_{+} = r^{-2} \big( dr^2 + \gamma_r \big)
\end{align*}
satisfies
\begin{align} \label{odd}
\mathrm{Ric}(g_{+}) + n g_{+} = O(r^{\infty})
\end{align}
when $n$ is odd, and
\begin{align} \label{even}
\mathrm{Ric}(g_{+}) + n g_{+} = O(r^{n-2})
\end{align}
when $n$ is even.  The metric $\gamma_r$ is given by a formal power series
\begin{align} \label{FPS}
\gamma_r = g_M + k^{(2)} r^2 + \cdots
\end{align}
in even powers of $r$ up to order $n-1$ when $n$ is odd, and up to order $n-2$ when $n$ is even.  Moreover, the coefficients in this range are determined by $g_M$, and obtained by
differentiating (\ref{odd}) (or (\ref{even})) and evaluating at $r =0$.  Up to a diffeomorphism fixing $M$, when $n$ is odd  there is in fact a unique formal power series solution of (\ref{odd}).  When $n$ is even, formal power series exist but they are not
unique (even modulo diffeomorphisms); see Theorem 2.3 of \cite{FG}.

Applying the Fefferman-Graham result to our setting, we conclude the following:  When the dimension $n$ is odd, there are tensors $k^{(2)}, k^{(4)}$ determined by $g_M$ such that the metric
\begin{align} \label{FPS2}
\widetilde{g} = r^{-2} \big( dr^2 + g_M + k^{(2)} r^2 + k^{(4)} r^4 \big)
\end{align}
satisfies
\begin{align} \label{app1}
\mathrm{Ric}(\widetilde{g}) + n \widetilde{g} = O(r^6).
\end{align}
The same holds when $n \geq 6$ is even.  When $n=4$, the coefficient $k^{(4)}$ in (\ref{FPS2}) is not determined by $g_M$, but one can choose such a tensor so that (\ref{app1}) holds.

To complete the proof of the lemma, for $0 < r \leq \epsilon$ we let
\begin{align*}
t = \log \frac{\epsilon}{r}.
\end{align*}
Then we can rewrite the metric in (\ref{FPS2}) as
\begin{align*}
\widetilde{g} = dt^2 + e^{2t} \epsilon^{-2}  g_M + k^{(2)} + e^{-2t} \epsilon^2 k^{(4)},
\end{align*}
which holds on $M \times [0,\infty)$.  Also, by (\ref{app1}),
\begin{align*}
\mathrm{Ric}(\widetilde{g}) + n \widetilde{g} = O (\epsilon^6 e^{-6t}).
\end{align*}
Taking $g_{\epsilon}' = \widetilde{g}$, the estimate (\ref{app}) follows.
 \end{proof}

\begin{rmk} Since $g_{\epsilon}$ and $g_{\epsilon}'$ are uniformly
equivalent, we can use either to measure norms defined above.
\end{rmk}

To slightly rephrase our goal in light of the preceding, we want to
find a symmetric $2$-tensor $h \in C^{2,\alpha}_1$ with sufficient decay at infinity so that
\begin{align} \label{modgoal}
\mathrm{Ric}(g_{\epsilon}' + h) + n(g_{\epsilon}' + h) = 0.
\end{align}
 The next issue we address is the well known lack of ellipticity
of the linearization of this equation.  We overcome this by using the standard technique of modifying by a
`gauge-fixing' term.  To explain this we need to introduce some notation.

For metrics $g$ and $\tilde{g}$ define the mapping
\begin{align} \label{Ndef}
\mathcal{N}_{\tilde{g},g}[h] = \mathrm{Ric}(\tilde{g} + h) + n (\tilde{g} + h) + \delta^*_{\tilde{g} + h} \beta_{g}(h),
\end{align}
where
\begin{align} \label{beta}
  \beta_{g}(h)_j = -(\nabla_g)^i h_{ij} + \frac{1}{2}
  (\nabla_g)_j ( \mathrm{tr}_{g} h)
\end{align}
is the {\em Bianchi operator}, and
\begin{align} \label{dual}
\delta^*_{\tilde{g}+h}(\omega)_{ij} = \frac{1}{2} \big( \nabla_{\tilde{g}+h,i} \omega_j + \nabla_{\tilde{g}+h,j} \omega_j \big)
\end{align}
is the $L^2$-adjoint of the divergence operator.  We also let
\begin{align} \label{LL}
L_{\tilde{g},g}(h) = \frac{d}{ds} \mathcal{N}_{\tilde{g},g}[ sh] \big|_{s=0}
\end{align}
denote the linearization of $\mathcal{N}$ at $h=0$.  It follows that
\begin{align*}
L_{\tilde{g},g}(h) = (D\mathrm{Ric}_{\tilde{g}} + n)h + \delta^*_{\tilde{g}} \beta_{g}(h),
\end{align*}
where $D\mathrm{Ric}$ denotes the linearization of the Ricci tensor.
From standard formulas (see e.g. Besse \cite{Besse}) we have
\begin{align} \label{Lform}
L_{\tilde{g},g}(h) = -\frac{1}{2} \Delta_{\tilde{g}} h + \mathfrak{D}_{\tilde{g}}(h) + nh + \mathcal{R}_{\tilde{g}}(h),
\end{align}
where $\mathfrak{D}_{\tilde{g}}$ is given by
\begin{align} \label{Dform}
\mathfrak{D}_{\tilde{g}}(h) = \delta^*_{\tilde{g}} \big\{  \beta_{\tilde{g}}(h) -  \beta_{g}(h)  \big\},
\end{align}
and $\mathcal{R}_{\tilde{g}}$ is given by
\begin{align} \label{Rterm}
\mathcal{R}_{\tilde{g}}(h)_{jk} = -\tilde{R}^{a\,\,b}_{\,\,j\,\,k} h_{ab} + \frac{1}{2}(\tilde{R}^a_{\,\, k } h_{aj} + \tilde{R}^a_{\,\, j}h_{ak}),
\end{align}
in terms of the curvature of $\tilde{g}$.   Notice that if $\tilde{g} = g$ (or more generally, if $\tilde{g}-g$ is sufficiently small) then the linearized operator
is elliptic.

\begin{rmk}  Although it will slightly complicate the argument in certain parts, overall it is much easier to work with the Bianchi operator
 with respect to the metric $g_{\epsilon}$ (instead of $g_{\epsilon}'$) when defining the gauge-fixing term.  As we will see below (Lemma \ref{prop:GrahamLee}), the boundary condition will also be defined in terms of $g_{\epsilon}$.  \end{rmk}


With this notation we can now reformulate our goal: to find a solution of
\begin{align}  \label{eq:gaugefixed}
\mathcal{N}_{g_{\epsilon}',g_{\epsilon}}[h] = 0.
\end{align}
In contrast to (\ref{modgoal}) the linearized operator $L_{g_{\epsilon}',g_{\epsilon}}$ is now elliptic, since $g_{\epsilon}' - g_{\epsilon}$ is small when $\epsilon > 0$ is small.
Unfortunately, it is not necessarily surjective, so we need
to allow additional variations of the metric $g_\epsilon'$.  We
therefore consider the following modification of \eqref{eq:gaugefixed}:
\begin{equation} \label{eq:gaugefixed2}
(r,h) \mapsto \mathcal{N}_{g_{\epsilon}' + r,g_{\epsilon}}[h] = \mathrm{Ric}(g_\epsilon' + r + h) + n(g_\epsilon' + r + h) +
\delta^*_{g_\epsilon'+r+h} \beta_{g_\epsilon}(h),
\end{equation}
where $r$ will be chosen in a suitable finite-dimensional
space to compensate for the lack of surjectivity of $L_{g_\epsilon', g_\epsilon}$.

We also need to verify that a zero of the mapping in (\ref{eq:gaugefixed2}) defines an Einstein metric.
The following result is a boundary-value version of Lemma 2.2 of
\cite{GL}), and as a byproduct it also specifies the boundary condition we will impose:

\begin{lem}\label{prop:GrahamLee}
  Suppose that $(r,h)$ is a zero of the mapping in \eqref{eq:gaugefixed2}
  with $r,h \in C^{2,\alpha}_1$ small enough so that $g_{+} = g_\epsilon' + h + r$ defines a Riemannian metric in $M \times [0, \infty)$.
  Assume  \\

  \noindent $(i)$  On the boundary $\{t=0\}$, we have
  \begin{align} \label{Bianchi}
  \beta_{g_\epsilon}(h) = 0.
  \end{align}

 \noindent $(ii)$ For some $K < 0$, $\mathrm{Ric}(g_{+}) \leq Kg_{+}.$  \\

  Then $\beta_{g_\epsilon}(h) = 0$ on
  $M\times [0,\infty)$, and hence (by (\ref{eq:gaugefixed2})) $g_{+}$ is a Poincar\'e-Einstein metric.
\end{lem}

\begin{proof}
  We let $\omega = \beta_{g_\epsilon}(h)$. Applying the Bianchi identity to
  \eqref{eq:gaugefixed2}, we obtain
  \[ \beta_{g_{+}}(\delta_{g_{+}}^{*}\omega) = 0. \]
  As in \cite{GL}, this implies
  \[ \Delta_{g_{+}} |\omega|^2 \geq -K |\omega|^2. \]
  Since $\omega=0$ on the boundary $\{t=0\}$, and $\omega\to 0$ as
  $t\to\infty$, the maximum principle implies that $\omega=0$
  everywhere.
\end{proof}

We are thus led to studying the linearization of the mapping in
\eqref{eq:gaugefixed2}, subject to the boundary condition
$\beta_{g_\epsilon}(h)|_{t=0} =0$.  Using (\ref{Lform}), the linearization of (\ref{eq:gaugefixed2}) is given by
\[ \begin{aligned}
    \overline{L}_{g_\epsilon', g_{\epsilon}}: E \times (C^{2,\alpha}_1)_\beta &\to
    C^{0,\alpha}_1 \\
    (r,h)  &\mapsto (D\mathrm{Ric}_{g_\epsilon'} + n)r +
    L_{g_\epsilon', g_{\epsilon}}(h),
\end{aligned} \]
where $E$ is a
certain finite dimensional subspace of $C^{2,\alpha}_1$, to be
determined later, and $(C^{2,\alpha}_1)_\beta$ denotes the space of
symmetric two tensors $h\in C^{2,\alpha}_1$ satisfying the boundary condition $\beta_{g_\epsilon}(h)|_{t=0} =0$.

 Most of our work in the paper will be
constructing a right inverse for this linearized operator, leading to
the following, proved in Section~\ref{sec:invertingLbar}.

\begin{thm} \label{thm:Lbarinverse}
  Let $\epsilon, \alpha > 0$ be sufficiently small.
  If the metric $g_M$ is chosen generically in its conformal class,
  then for a suitable finite dimensional subspace $E$ the linearized
  operator $\overline{L}_{g_\epsilon'}$ has a right inverse
  $\mathcal{R}$, satisfying $\Vert \mathcal{R} \Vert \leq
  C\epsilon^{-2-\alpha}$ for a constant $C$ independent of $\epsilon$.
\end{thm}

Using this result together with Lemma~\ref{lem:errorestimate}, a
standard contraction mapping argument can be used to solve
Equation~\ref{eq:gaugefixed2}, as follows. Let us define the operator
$\mathcal{Q}$ by
\begin{equation}\label{eq:Qdefn}
\begin{aligned}
 \mathrm{Ric}(g_\epsilon'+r+h) + n(g_\epsilon'+r+h)
    &+ \delta^*_{g_\epsilon'+r+h} \beta_{g_\epsilon}(h)\\ &= \mathrm{Ric}(g_\epsilon') + ng_\epsilon' +
  \overline{L}_{g_\epsilon', g_{\epsilon}}(r,h) + \mathcal{Q}(r,h),
\end{aligned}
\end{equation}
and define $\mathcal{F}$ by
\[ \begin{aligned}
    \mathcal{F} : E \times (C^{2,\alpha}_1)_\beta &\to E\times (C^{2,\alpha}_1)_\beta
    \\
    (r, h) &\mapsto -\mathcal{R}\Big[ \mathrm{Ric}(g_\epsilon') +
    ng_\epsilon' + \mathcal{Q}(r,h)\Big].
\end{aligned} \]
A fixed point of $\mathcal{F}$ then necessarily satisfies
Equation~\eqref{eq:gaugefixed2}.

Define the set
\[ \mathcal{U} = \{(r,h)\in E \times (C^{2,\alpha}_1)_\beta\,:\,
  \Vert(r,h)\Vert \leq \epsilon^3 \}, \]
using the norm
\[ \Vert (r,h)\Vert = \Vert r \Vert_{C^{2,\alpha}_1} + \Vert
  h\Vert_{C^{2,\alpha}_1}. \]
\begin{prop}
  For sufficiently small $\epsilon, \alpha$
  the map $\mathcal{F}$ defines a contraction $\mathcal{F} :
  \mathcal{U} \to \mathcal{U}$, and so it has a fixed point.
\end{prop}
\begin{proof}
First note that by differentiating Equation~\eqref{eq:Qdefn} with
respect to $g_\epsilon'$ and applying the mean value theorem (or
alternatively expanding $\mathcal{Q}$ as a power series), we find that
as long as $\Vert(r_1, h_1)\Vert, \Vert(r_2,h_2)\Vert < \kappa < c_0$
for a fixed constant $c_0$, we have
\[ \Vert \mathcal{Q}(r_1,h_1) -
  \mathcal{Q}(r_2,h_2)\Vert_{C^{0,\alpha}_1} \leq C\kappa \Vert
  (r_2-r_1, h_2-h_1)\Vert. \]
Using our bound for the right inverse $\mathcal{R}$, it follows that
as long as $(r_i,h_i)\in \mathcal{U}$, and $\epsilon$ is sufficiently
small, we have
\[ \Vert \mathcal{F}(r_1,h_1) - \mathcal{F}(r_2,h_2)\Vert \leq
  C\epsilon^{1-\alpha} \Vert (r_1-r_2, h_1-h_2)\Vert, \]
and so $\mathcal{F}$ is a contraction.

Finally to check that $\mathcal{F}(\mathcal{U}) \subset \mathcal{U}$
we let $(r,h)\in \mathcal{U}$. Then
\[ \begin{aligned} \Vert \mathcal{F}(r,h)\Vert &\leq \Vert
    \mathcal{F}(r,h) - \mathcal{F}(0,0)\Vert + \Vert
    \mathcal{F}(0,0)\Vert \\
    &\leq C\epsilon^{1-\alpha}\Vert (r,h)\Vert + C\epsilon^{4-\alpha} \\
    &\leq \epsilon^3
\end{aligned} \]
for sufficiently small $\epsilon$. Here we used that by
Lemma~\ref{lem:errorestimate} and the bound for $\mathcal{R}$ we have
$\Vert \mathcal{F}(0,0)\Vert \leq C\epsilon^{-2-\alpha}\epsilon^6$.
\end{proof}

The existence of a fixed point of $\mathcal{F}$ together with
Proposition~\ref{prop:GrahamLee} then completes the proof of
Theorem~\ref{thm:mainthm}. In the remainder of this section we give a
brief outline of the proof of Theorem~\ref{thm:Lbarinverse}.

The first step, in Section~\ref{SecHyp} is to carefully analyze the linearzed operator
$L_{\mathbf{H}} = L_{g_{\mathbf{H}}, g_{\mathbf{H}}}$ in the model case when $g_{\mathbf{H}}$ is the
hyperbolic metric, i.e. $M=\mathbf{R}^n$ and $g_M$ is the Euclidean
metric. The main result here is Theorem~\ref{thm:Hinverse} below,
which roughly speaking says the following: given a 2-tensor $u$ supported inside
the unit ball in the spatial direction, and satisfying an additional
``orthogonality condition'' $I(u)=0$, we can solve
$L_{g_{\mathbf{H}}}(h)=u$ with $h$ satisfying the Bianchi boundary
condition (with respect to $g_{\mathbf{H}}$), such that $h$ is localized in the sense that it has good
decay in the spatial directions. Here $I(u)$ is a one-form on $M$, see
\eqref{Idef} for its definition.

To illustrate this, consider the following simple analogous result. Let
$\Delta$ be the Laplacian on the product space
$\mathbf{R}^n \times X$ for a compact Riemannian manifold $X$, and let
$u$ be a function supported in $B_1\times X$. We can then construct a
solution of $\Delta h = u$ with $h$ decaying at
the rate of the Green's function $r^{2-n}$ for large $r$ in the
$\mathbf{R}^n$ direction. If, however, we impose the additional condition that
$u$ is orthogonal to the constants in each fiber $\{t\} \times X$,
then we can find a solution $h$ decaying exponentially fast.

The next step is to globalize this result to the case when $M$ is a
compact manifold. The idea is that when $\epsilon$ is sufficiently
small, then locally $(M,\epsilon^{-2} g_M)$ is well approximated by
Euclidean space. We can then solve the equation $L_{g_{\epsilon}',g_{\epsilon}}h=u$
on $M\times [0,\infty)$ as long as $u$ satisfies the orthogonality
condition $I(u)=0$, by chopping $u$ up into pieces supported in
approximately Euclidean balls, and combining the ``local'' inverses
constructed in the model space. The decay of the corresponding local
solutions ensures that we get a good estimate for the error obtained
from combining these local solutions. We need some additional steps to
ensure that after this cutting and pasting procedure we can still
impose the Bianchi condition.

It remains to deal with the case when $I(u)\ne 0$. Since $I(u)$ is a
one-form on $M$, we are able to reduce this to inverting a suitable
linear operator on $M$. More precisely, we consider the operator
\[ \begin{aligned} \mathcal{T} : C^{2,\alpha}(\Omega^1(M)) &\to
  C^{0,\alpha}(\Omega^1(M)) \\
  \omega &\mapsto I\circ L_{g_{\epsilon}',g_{\epsilon}}( \epsilon^{-2} e^{-nt}
  \omega\odot dt)
\end{aligned} \]
It turns out that $\mathcal{T}$, which depends on $\epsilon$, converges to an elliptic
operator $\mathcal{T}_0$ as $\epsilon\to 0$, but $\mathcal{T}_0$ is
not necessarily surjective. It is this issue that we overcome by
incorporating an additional finite dimensional space $E$ of symmetric
2-tensors on $M \times [0,\infty)$ in the problem, and instead we consider the operator
\[ \begin{aligned} \overline{\mathcal{T}} : E \times
    C^{2,\alpha}(\Omega^1(M)) &\to C^{0,\alpha}(\Omega^1(M)) \\
    (r, \omega) &\mapsto I\Big[ (D\mathrm{Ric}_{g_\epsilon'} + n)r + L_{g_{\epsilon}',g_{\epsilon}}(\epsilon^{-2}e^{-nt}
    \omega\odot dt)\Big]
\end{aligned} \]
Although this operator is not elliptic in $r$, we only need a finite
dimensional space $E$ since the cokernel of $\mathcal{T}$ is finite
dimensional. It turns out that as long as $g_M$ admits no Killing
vector fields, we can choose a finite dimensional space $E$ such that
$\overline{\mathcal{T}}$ is surjective. This is then enough to
construct the right inverse required in
Theorem~\ref{thm:Lbarinverse}.

\section{The linearized operator on Hyperbolic space}  \label{SecHyp}

In this section we study the linearized operator $L_{g_{\mathbf{H}}} = L_{g_{\mathbf{H}},g_{\mathbf{H}}}$ in \eqref{Lform}
on hyperbolic space $\mathbf{H}^{n+1}$ with the hyperbolic metric
$g_{\mathbf{H}}$. A standard calculation gives
\begin{align} \label{LH}
L_{g_{\mathbf{H}}} h = - \frac{1}{2}\Delta_{g_{\mathbf{H}}} h - h + (\mathrm{tr}_{g_{\mathbf{H}}} h)g_{\mathbf{H}}.
\end{align}

A basic result (see \cite{LeeMemoirs}, Theorem 5.9) is the following:

\begin{thm} \label{LeeThm} On hyperbolic space $\mathbf{H}^{n+1}$, the linearized
  operator $L= L_{g_\mathbf{H}}$ at the hyperbolic metric is an isomorphism $L :
  C^{k,\alpha}_\delta \to C^{k-2,\alpha}_\delta$, as long as $|\delta
  - n/2| < n/2$.  In particular, this
  holds for our choice of weight.
\end{thm}

The main technical result we will need is a variant Theorem \ref{LeeThm} solving a boundary value problem.  As above let
$\mathbf{H}^{n+1}_+ = \mathbf{R}^n \times[0,\infty)$, a subset of
hyperbolic space equipped with the hyperbolic metric
\[ g_{\mathbf{H}} = dt^2 + e^{2t} (dx^i)^2. \]
We will sometimes write $x^0 = t$.   Indices $i,j,k,l,\ldots$ run from 1 to $n$, while indices
$a,b,c,\ldots$ run from 0 to $n$.

For a symmetric $2$-tensor $u$ on $\mathbf{H}^{n+1}_+$ define the
one-form $I : T \mathbf{R}^n \rightarrow \mathbf{R}$ on $\mathbf{R}^n$ by
\begin{align} \label{Idef}
I(u)(V) = \int_0^\infty u(V, \partial_t) e^{-2t}\, dt
\end{align}
where $V \in T \mathbf{R}^n$.  More generally, given a manifold $M$ and a symmetric $2$-tensor $u \in C^{0,\alpha}_{\delta}(M \times [0, \infty))$, then (\ref{Idef}) defines a one-form
$I(u)$ on $M$ as long as $\delta > -2$.

\begin{thm}\label{thm:Hinverse}
  Suppose that $u \in C^{0,\alpha}_1$
  is a symmetric two-tensor on $\mathbf{H}^{n+1}_+$ supported in $B_1 \times
  [0,\infty)$, with $I(u)=0$. Then there exists a symmetric two-tensor
  $h \in C^{2,\alpha}_1$ on
  $\mathbf{H}^{n+1}_+$ satisfying
  \begin{enumerate}
    \item $Lh = u$, and $\Vert h\Vert_{C^{2,\alpha}_1}\leq C\Vert
      u\Vert_{C^{0,\alpha}_1}$ for a uniform constant $C$.
    \item $\beta_{g_{\mathbf{H}}}(h) = 0$ along the boundary
      $\{t=0\}$,
    \item For any $\delta\in (0,1)$, $h$ decays in the $x^i$ directions, at a rate of at least
      $|x|^{-n-1+\delta}$. More precisely, let $A_{R-1,R} = (B_R\setminus
      B_{R-1}) \times [0,\infty)$. We have
      \[ \Vert h\Vert_{C^{2,\alpha}_1(A_{R-1,R})} \leq C
      R^{-n-1+\delta} \Vert u\Vert_{C^{0,\alpha}_1} \]
      for all $R > 1$,
      for a uniform constant $C$.
  \end{enumerate}
   We define the linear operator $P_{\mathbf{H}}$ by setting
   $P_{\mathbf{H}}(u) = h$.
\end{thm}


The rest of this section will be devoted to the proof of Theorem \ref{thm:Hinverse}.  Since it is rather involved, we begin with a sketch.

Given a symmetric $2$-tensor $u$ as in the statement of the theorem, the first step is to construct a solution $h_0$ of
\begin{align} \label{h0}
L h_0 = u
\end{align}
on $[0,\infty) \times M$, using the Green's function of $L$.  Note
that this solution will not in general satisfy the Bianchi condition
$\beta_{g_{\mathbf{H}}}h_0 = 0$ on the boundary $\{ t = 0 \}$.
Therefore, we need to `correct' our solution by solving the homogeneous boundary-value problem
\begin{align} \label{homog}
\begin{cases}
L h_1 = 0\ \mbox{in } \mathbf{H}_{+}^{n+1}, \\
\beta_{g_{\mathbf{H}}}(h_1) = \beta_{g_{\mathbf{H}}}(h_0)\ \mbox{on } \partial \mathbf{H}_{+}^{n+1} = \{ t = 0 \} \times \mathbf{R}^n,
\end{cases}
\end{align}
where $h_0$ solves (\ref{h0}).  Then taking $h = h_0 - h_1$, we arrive
at a solution of the original problem. We will solve the homogeneous
problem using the Fourier transform, and analyzing the resulting
ODEs. The required decay in Theorem~\ref{thm:Hinverse} will be
obtained by controlling the singularity of the Fourier transform at
the origin, and the orthogonality condition $I(u)=0$ is used to ensure
that the terms with the worst singularity vanish, thereby improving
the decay of the solution.

\subsection{The Fourier transform of the homogeneous problem}
We begin by writing down explicit formulas for the components of $L h$
for a symmetric 2-tensor $h$
with respect to the coordinates $x^i, t$. We will write $x^0=t$, and
use the convention that indices $i,j,k,\ldots$ run from 1 to $n$,
while $a,b,c,\ldots$ run from 0 to $n$.

\begin{lem} \label{CSym} With respect to the basis $\{ \partial_{x^0},
  \ldots, \partial_{x^n}\}$, the only nonzero Christoffel symbols are
\begin{align*}
 \Gamma^0_{jk} &= -e^{2t}\delta_{jk} ,\\
 \Gamma^i_{0k} &= \Gamma^i_{k0} = \delta^i_k,
\end{align*}
where $\nabla$ is the Riemannian connection.

More generally, if $(M,g_M)$ is a Riemannian manifold, $g = dt^2 + e^{2t} g_M$ is a warped product metric, and  $\{ x^i \}$ are local coordinates on $M$, then the only non-zero Christoffel symbols
with respect to the coordinate system $\{ x^1, \dots, x^n, x^0 = t \}$ on $M \times [0, \infty)$ are
 \begin{align*}
 \Gamma^m_{jk} &= (\Gamma_M)^m_{jk}, \\
 \Gamma^0_{jk} &= -e^{2t} ( g_M )_{jk} ,\\
 \Gamma^i_{0k} &= \Gamma^i_{k0} = \delta^i_k,
\end{align*}
where $\Gamma_M$ are the Christoffel symbols with respect to $g_M$.
\end{lem}

This is a straightforward calculation, and we will omit the proof.  Using these formulas, we have the following identities for the components of the covariant derivatives of a symmetric two-tensor:
\[ \begin{aligned} \label{DH}
\nabla_i h_{jk} &= \partial_i h_{jk} + e^{2t} \delta_{ij} h_{0k} +
e^{2t} \delta_{ik} h_{0j} \\
\nabla_0 h_{jk} &= \partial_0 h_{jk} - 2h_{jk} \\
\nabla_i h_{j0} &= \partial_i h_{j0} + e^{2t} \delta_{ij} h_{00} -
h_{ij}\\
\nabla_0 h_{j0} &= \partial_0 h_{j0} - h_{j0}\\
\nabla_i h_{00} &= \partial_i h_{00} - 2h_{i0}\\
\nabla_0 h_{00} &= \partial_0 h_{00}.
\end{aligned} \]
Using these formulas we can compute the Bianchi operator:
\[ \beta(h)_a = g^{bc}\nabla_b h_{ac} - \frac{1}{2} \nabla_a
(g^{bc}h_{bc}). \]
Its components are
\begin{align} \label{Bianh} \begin{split}
  \beta(h)_i &= e^{-2t} \partial_j h_{ij} + \partial_0 h_{i0} +
  n h_{i0} - \frac{1}{2} e^{-2t} \partial_i h_{kk} -
  \frac{1}{2} \partial_i h_{00} \\
  \beta(h)_0 &= e^{-2t} \partial_i h_{i0} + n h_{00} +
  \frac{1}{2}\partial_0 h_{00} - \frac{1}{2} e^{-2t} \partial_0
  h_{ii}.
  \end{split}
\end{align}
We can also take another covariant derivative and compute the components of the rough laplacian acting on symmetric $2$-tensors $\Delta  = g_{\mathbf{H}}^{ab} \nabla_a \nabla_b$:
\begin{align*}
\Delta h_{jk} &= e^{-2t} \partial_i\partial_i h_{jk} + \partial_t^2 h_{jk} +
(n-4)\partial_t h_{jk} + (2-2n) h_{jk} \\
&\quad + 2 e^{2t} h_{00} \delta_{jk} + 2(\partial_j h_{0k} + \partial_k
h_{0j}), \\
\Delta h_{j0} &= e^{-2t} \partial_i\partial_i h_{j0} + \partial_t^2 h_{j0} +
(n-2) \partial_t h_{j0} - 2(n + 1) h_{j0} \\
&\quad + 2\partial_j h_{00} - 2e^{-2t} \partial_i h_{ij}\\
\Delta h_{00} &= e^{-2t} \partial_i\partial_i h_{00} + \partial_t^2 h_{00} + n \partial_t
h_{00} - 2n h_{00}\\
&\quad - 4e^{-2t}\partial_i h_{0i}  + 2 e^{-2t} h_{kk}.\\
\end{align*}

Combining the above, we can write the equation $Lh = u$ as a system of
equations in the components of $h$ and $u$:
\begin{align} \label{Lhomog} \begin{split}
u_{jk} &= e^{-2t} \partial_i\partial_i h_{jk} + \partial_t^2 h_{jk} +
(n-4)\partial_t h_{jk} + (4-2n) h_{jk} \\
&\quad - 2h_{ii} \delta_{jk} + 2(\partial_j h_{0k} + \partial_k
h_{0j}) \\
u_{j0} &= e^{-2t} \partial_i\partial_i h_{j0} + \partial_t^2 h_{j0} +
(n-2) \partial_t h_{j0} - 2nh_{j0} \\
&\quad + 2\partial_j h_{00} - 2e^{-2t} \partial_i h_{ij} \\
u_{00} &= e^{-2t} \partial_i\partial_i h_{00} + \partial_t^2 h_{00} + n \partial_t
h_{00} - 2n h_{00} \\
&\quad - 4e^{-2t}\partial_i h_{0i}.
\end{split}
\end{align}


In the following, we will use upper-case letters to denote the Fourier
transforms of components of $h$, scaled by
additional powers of $e^{-t}$. This amounts to writing our tensor $h$ in
terms of an orthonormal frame, and it leads to an ODE system which is
easier to analyze. With this in mind we define
\begin{equation}  \label{FT}
\begin{aligned}
H_{ij}(t,\xi) &= e^{-2t} \widehat{h_{ij}}(t, \xi) = e^{-2t} \int_{\mathbf{R}^n} e^{- \sqrt{-1} \xi \cdot
  \mathbf{x}} h_{ij}(t,\mathbf{x}) d\mathbf{x}, \\
H_{i0}(t,\xi) &= e^{-t} \widehat{h_{i0}}(t,\xi) = e^{-t} \int_{\mathbf{R}^n} e^{- \sqrt{-1} \xi \cdot
  \mathbf{x}} h_{i0}(t,\mathbf{x}) d\mathbf{x}, \\
H_{00}(t,\xi) &= \widehat{h_{00}}(t,\xi) =  \int_{\mathbf{R}^n} e^{- \sqrt{-1} \xi \cdot
  \mathbf{x}} h_{00}(t,\mathbf{x}) d\mathbf{x},
\end{aligned}
\end{equation}
and similarly we will write $U_{ij} = e^{-2t}\widehat{u_{ij}}$, etc.
After applying the Fourier transform to the system (\ref{Lhomog}), we obtain the following system of ODEs:
\begin{align} \label{ODEh} \begin{split}
 \Big[H_{jk}'' + nH_{jk}' - 2\delta_{jk} H_{pp}\Big] - 2\sqrt{-1} e^{-t} ( \xi_j H_{0k} + \xi_k H_{0j} )
 - e^{-2t} |\xi|^2 H_{jk} &= U_{jk} \\
 \Big[H_{j0}'' +nH_{j0}' - (n+1) H_{j0}\Big] -2e^{-t}\sqrt{-1}(\xi_j H_{00} - \xi_i
H_{ij})- e^{-2t} |\xi|^2 H_{j0} &= U_{j0}\\
 \Big[H_{00}'' + n H_{00}' - 2nH_{00}\Big] + 4\sqrt{-1} e^{-t} \xi_i H_{0i}
- e^{-2t} |\xi|^2 H_{00} &= U_{00},
\end{split}
\end{align}
and we are for now interested in the case when $U=0$.

Applying the Fourier transform to the components of the Bianchi operator in (\ref{Bianh}) gives
\begin{align} \label{BianH} \begin{split}
B_\xi(H)_i &= e^{-t}\widehat{\beta(h)_i} =  H_{i0}' + (n+1)H_{i0}
-e^{-t} \sqrt{-1}\xi_j H_{ij} \\
&\qquad\qquad\qquad + \frac{1}{2} e^{-t} \sqrt{-1}\xi_i H_{pp} +
\frac{1}{2} e^{-t} \sqrt{-1} \xi_iH_{00} \\
B_\xi(H)_0 &=  \widehat{\beta(h)_0} = n H_{00} + \frac{1}{2}H_{00}' -
  \frac{1}{2}H_{pp}' - H_{pp} - e^{-t} \sqrt{-1} \xi_i H_{i0}.
  \end{split}
\end{align}

\subsection{Solutions for small $\xi$}
We will assume that $|\xi|$ is small, and find solutions of the system
of ODEs as perturbations of solutions to the simpler system when
$\xi=0$, as a power series in $\xi, \bar\xi$. Let us write the ODEs
\eqref{ODEh} with $U=0$ as $L_\xi(H)=0$. If we write
\begin{equation}\label{eq:Hseries}
 H(\xi, t) = H(0, t) + \xi_i \partial_{\xi_i} H(0,t) +
  \bar\xi_i \partial_{\bar\xi_i} H(0,t) + \ldots,
\end{equation}
then we can obtain equations satisfied by $H(0,t)$ and $\partial_{\xi_i}
H(0,t)$ by differentiating the equation $L_\xi(H)=0$ and setting
$\xi=0$. In particular, $H(0,t)$ satisfies
\begin{align} \label{ODEh2} \begin{split}
 H_{jk}'' + nH_{jk}' - 2\delta_{jk} H_{pp} &= 0 \\
 H_{j0}'' +nH_{j0}' - (n+1) H_{j0}  &= 0\\
 H_{00}'' + n H_{00}' - 2nH_{00} &= 0.
\end{split}
\end{align}
$\overline{\partial}_{\xi_i} H(0,t)$ also satisfies the
same equations, while $\partial_{\xi_i} H(0,t)$ satisfies
\begin{align} \label{ODEh3} \begin{split}
 \Big[\partial_{\xi_l} H_{jk}'' + n\partial_{\xi_l} H_{jk}' -
 2\delta_{jk} \partial_{\xi_l} H_{pp}\Big] - 2\sqrt{-1} e^{-t} ( \delta_{lj}
 H_{0k} + \delta_{lk} H_{0j} ) &= 0 \\
 \Big[\partial_{\xi_l} H_{j0}'' +n \partial_{\xi_l} H_{j0}' - (n+1) \partial_{\xi_l} H_{j0}\Big] -2e^{-t}\sqrt{-1}(\delta_{lj} H_{00} - \delta_{li}
H_{ij}) &= 0\\
 \Big[\partial_{\xi_l} H_{00}'' + n \partial_{\xi_l} H_{00}' -
 2n \partial_{\xi_l} H_{00}\Big] + 4\sqrt{-1} e^{-t} \delta_{li}
 H_{0i} &= 0.
\end{split}
\end{align}
For higher order derivatives $\partial_{\xi}^k$, we will have a system
that we write schematically as
\begin{equation}\label{eq:partialkH}
 L(\partial_{\xi}^k H) + e^{-t} \ast \partial_{\xi}^{k-1}H + e^{-2t}
  \ast \partial_{\xi}^{k-2}H = 0,
\end{equation}
where $L$ is the homogeneous 2nd order operator appearing in square
brackets above. The solutions of the system for
$H, \partial_{\xi}H$ that we write down below will all be of order
$e^{-nt}$ and $e^{-(n+1)t}$ respectively, or smaller.
Because of the additional factors of $e^{-t}$,
it follows that the inhomogeneous equations
\eqref{eq:partialkH} for $\partial_{\xi}^kH$ has a solution of order
$e^{-(n+k)t}$, or smaller, and so the solution $H$ given by the series
\eqref{eq:Hseries} satisfies $H = O(e^{-nt})$ as $t\to\infty$.

\subsubsection{Solutions of type I}
Let $a_{ij}$ be any trace free symmetric matrix. Define
\[ \begin{aligned}
         H_{ij} &= \sqrt{-1} a_{ij} e^{-nt}, \\
         H_{i0} &= 0, \\
         H_{00} &= 0.
\end{aligned} \]
This solves the equations \eqref{ODEh2}. We can set
$\partial_{\bar\xi_l} H=0$, and also in \eqref{ODEh3} only the
equations involving $\partial_{\xi_l} H_{j0}$ are inhomogeneous. So we
can let $\partial_{\xi_l} H_{ij},\partial_{\xi_l}
H_{00} = 0$, while $\partial_{\xi_l} H_{j0}$ satisfies
\[ \Big[\partial_{\xi_l} H_{j0}'' +n \partial_{\xi_l} H_{j0}' -
  (n+1) \partial_{\xi_l} H_{j0}\Big] - 2e^{-t} a_{lj} e^{-nt}
= 0. \]
A solution of this ODE is
\[ \partial_{\xi_l} H_{j0} = -\frac{2}{n+2} a_{lj} t e^{-(n+1)t}. \]
We can similarly obtain solutions of the equations obtained by
differentiating $L_\xi(H)=0$ more than once, and solve them
inductively. The inhomogeneous terms
in these equations will all be of order $|\xi|^2e^{-(n+2)t}$ or
smaller. It follows that we can find a solution $H^1$ of our system
\eqref{ODEh} such that
\begin{align} \label{H1} \begin{split}
         H^1_{ij} &= \sqrt{-1} a_{ij} e^{-nt} + O(|\xi|^2 e^{-(n+2)t}), \\
         H^1_{i0} &= b_{ij} \xi_j t e^{-(n+1)t} + O(|\xi|^2 e^{-(n+2)t}), \\
         H^1_{00} &= O(|\xi|^2 e^{-(n+2)t}).
         \end{split}
\end{align}

Let $B_{\xi}$ denote the Fourier transform of the Bianchi operator,
i.e., the operator appearing on the RHS of (\ref{BianH}).  Then
\begin{align} \label{BH1} \begin{split}
B_{\xi}(H^1)_i &= e^{-(n+1)t} \xi_j a_{ij} \left[ \frac{-2}{n+2}
  + 1\right] + O(|\xi|^2 e^{-(n+2)t}), \\
B_{\xi}(H^1)_0 &= O(|\xi|^2 te^{-(n+2)t}).
\end{split}
\end{align}
Evaluating at $t=0$ we have
\[ \begin{aligned}
  B_\xi(H^1)_i|_{t=0} &= \xi_j a_{ij} \frac{n}{n+2} + O(|\xi|^2) \\
  B_\xi(H^1)_0|_{t=0} &= O(|\xi|^2).
\end{aligned} \]

\subsubsection{Solutions of type II, III}
We now let
\[ \lambda = \frac{n+\sqrt{n^2 + 8n}}{2}, \]
and note that $n + 1 < \lambda < n+2$.  For constants $a,b$, let us set
\[ \begin{aligned}
  H_{jk} &= ae^{-\lambda t}\delta_{jk}, \\
  H_{j0} &= 0,\\
  H_{00} &= be^{-\lambda t}.
\end{aligned} \]
These give a solution of \eqref{ODEh2}. Again, from
\eqref{ODEh3} only the equations for $\partial_{\xi_l}H_{j0}$ have
a nonzero inhomogeneous term:
\[ \Big[\partial_{\xi_l} H_{j0}'' +n \partial_{\xi_l} H_{j0}' -
  (n+1) \partial_{\xi_l} H_{j0}\Big] - 2\sqrt{-1} e^{-t} (b-a)\delta_{lj}
  e^{-\lambda t}
= 0. \]
A solution of this equation is
\[ \partial_{\xi_l} H_{j0} = K(b-a) \delta_{lj} e^{-(\lambda+1)t}, \]
where
\[K = \frac{2\sqrt{-1}}{\lambda^2 + (2-n)\lambda - 2n}. \]
As before, it follows that we can find a solution $\tilde{H}$ of
\eqref{ODEh} satisfying
\begin{align} \label{H23} \begin{split}
  \tilde{H}_{jk} &= ae^{-\lambda t} \delta_{jk} + O(|\xi|^2
  e^{-(\lambda+2)t}), \\
  \tilde{H}_{j0} &= K(b-a) \xi_j e^{-(\lambda+1)t} + O(|\xi|^2
  e^{-(\lambda+2)t}), \\
  \tilde{H}_{00} &= be^{-\lambda t} + O(|\xi|^2 e^{-(\lambda+2)t}).
\end{split}
\end{align}
Substituting these into \eqref{BianH} we find
\[ \begin{aligned}
  B_\xi(\tilde{H})_i|_{t=0} &= \left[ \left( (\lambda-n)K +
        \frac{n-2}{2}\sqrt{-1}\right) a + \left( (n-\lambda)K +
        \frac{1}{2}\sqrt{-1}\right)b \right] \xi_i  + O(|\xi|^2), \\
  B_\xi(\tilde{H})_0|_{t=0} &= \left(n - \frac{\lambda}{2}\right)b -
  \left(n -\frac{\lambda n}{2}\right) a + O(|\xi|^2).
\end{aligned} \]
Choosing $a,b$ suitably, we obtain two different solutions, $H^2, H^3$
of \eqref{ODEh}, satisfying
\[ \begin{aligned}
  B_\xi(H^2)_i|_{t=0} &= \xi_i + O(|\xi|^2), \\
  B_\xi(H^2)_0|_{t=0} &= O(|\xi|^2),
\end{aligned} \]
and
\[ \begin{aligned}
  B_\xi(H^3)_i|_{t=0} &= O(|\xi|^2), \\
  B_\xi(H^3)_0|_{t=0} &= 1 + O(|\xi|^2).
\end{aligned} \]

\subsubsection{Solutions of type IV}
With the same choice of $\lambda$ as above, set
\begin{align} \label{H4} \begin{split}
  H_{jk} &= e^{-\lambda t}\delta_{jk}, \\
  H_{00} &= e^{-\lambda t},\\
  H_{i0} &= b_i e^{-(n+1)t},
\end{split}
\end{align}
for arbitrary $b_1,\ldots, b_n$. This tensor satisfies
\eqref{ODEh2} and as above, we can iteratively solve inhomogeneous
ODEs for $\partial_{\xi}^kH$ to find a solution $H^4$ of \eqref{ODEh},
sastisfying $H^4 = O(e^{-nt})$. We will not need to know the value of
the Bianchi operator for these solutions.

\begin{lem}  \label{ModelSolns}  The solutions of types I, II,
III and IV together form an $(n+1)(n+2)/2$-dimensional space of solutions of
$L_\xi(H)=0$.  Moreover, all of them decay at a rate of at least $e^{-nt}$ as
$n\to\infty$. \end{lem}

\begin{proof} To explain the dimension count: the solutions of type I are in one-to-one correspondence with
trace-free symmetric $n \times n$ matrices; hence the dimension of this space of solutions is $n(n+1)/2 - 1$.  The solutions of type
II and III depend on two different choices of the parameter $a$, hence there is a two-dimensional space of these kinds of solutions.  Finally,
the set of solutions of type IV is obviously $n$-dimensional, since we can choose the vector $(b_1,\dots,b_n)$ arbitrarily.  Summing, we have
$[ n(n+1)/2 - 1 ] + 2 + n = (n+1)(n+2)/2.$  It is clear from the leading terms in (\ref{H1}), (\ref{H23}), and (\ref{H4}) that this family of solutions is linearly independent.
\end{proof}

\subsubsection{Prescribing the boundary condition for small
  $\xi$}\label{subsec:smallxi}
We can now combine the solutions $H^1, H^2, H^3$ that we obtained
above, to find that for any symmetric matrix $a_{ij}$ (not necessarily
trace free), and constant $a$, there is a solution of $L_\xi(H) = 0$
satisfying
\begin{equation}\label{eq:solnsmallxi}
 \begin{aligned}
  B_\xi(H)_i|_{t=0} &= a_{ij} \xi_j + O(|\xi|^2) \\
  B_\xi(H)_0|_{t=0} &= a + O(|\xi|^2).
\end{aligned}
\end{equation}
This solution $H$ is a smooth function of $\xi$, $a_{ij}$,
$a$, and in addition $H = O(e^{-nt})$.

\begin{lem} \label{smallsolutions} For each $1 \leq a \leq n+1$ and
  $\xi\ne 0$, we can find a solution $H^a$ (with the same
decay properties) satisfying
\begin{align*} B_\xi(H^a)|_{t=0} = \mathbf{e}_a, \end{align*}
where $\mathbf{e}_a \in \mathbf{R}^{n+1}$ is a standard basis vector.
\end{lem}

\begin{proof}  Define $a_{ij}$ to be the symmetric matrix such that $a_{1i} =
|\xi|^{-2} \xi_i$
for all $i$, and $a_{ii} = - |\xi|^{-2} \xi_1$ for $i = 2, \ldots, n$,
and $a_{ij} = 0$ for the other entries. Also, let
$a=0$. The corresponding solution $H$ satisfies
\[ \begin{aligned}
  B_\xi(H)_1|_{t=0} &= 1 + O(|\xi|), \\
  B_\xi(H)_i|_{t=0} &= O(|\xi|), \text{ for }i > 1  \\
  B_\xi(H)_0|_{t=0} &= O(|\xi|^2).
\end{aligned} \]

We can repeat this construction replacing the index $1$ with any $j >
1$, and finally we can also set $a_{ij}=0, a=1$. In this way, for any
standard basis vector $\mathbf{e}_a \in \mathbf{R}^{n+1}$ we can obtain a
solution $\tilde{H}^a$ satisfying
\[ B_\xi({\tilde{H}}^a)|_{t=0} = \mathbf{e}_a + O(|\xi|). \]
For sufficiently small $\xi$, say $|\xi| < \kappa$,
we can then take linear combinations
\[ H^a = \lambda_a \tilde{H}^a + \sum_{b\ne a} \lambda_b
\tilde{H}^b, \]
where $\lambda_a = 1 + O(|\xi|)$ and $\lambda_b = O(|\xi|)$ for $b\ne
a$, and $H^a$ will satisfy
\[ B_\xi(H^a)|_{t=0} = \mathbf{e}_a. \]
\end{proof}

The key question for us is the nature of the singularity of these solutions
$H^a$ at $\xi=0$. From the preceding discussion we see that
the components of each $H^a$ have the form
\begin{equation} \label{eq:Habc}
 H^a_{bc} = |\xi|^{-2} \Phi^a_{bc}(\xi,t),
\end{equation}
where the $\Phi^a_{bc}$ are smooth functions of $\xi,t$ satisfying
$\Phi^a_{bc}(0,t)=0$ and $\Phi^a_{bc}(\xi, t) = O(e^{-nt})$.

\subsection{Solutions for large $\xi$}
Consider again the ODEs \eqref{ODEh}, satisfied by the Fourier
transform $H$ of a solution of $Lh=0$. We now study solutions of this
system for large $\xi$, with the aim of prescribing $\beta(h)$ at
$t=0$. The following simple observation shows that this is equivalent to
studying solutions of the system with $|\xi|=1$, but $t \to
-\infty$.
\begin{lem}\label{lem:translate}
  Suppose that $H(\xi, t)$ is a solution of the system
  \eqref{ODEh}. Then for any $T \in \mathbf{R}$ another solution is
  given by $\tilde{H}(\xi, t) = H(e^T \xi, t + T)$. In addition,
  applying the Fourier transform of the Bianchi operator, we have
  \[ B_\xi(\tilde{H})|_{t=0} = B_{e^T\xi}(H)|_{t=T}. \]
\end{lem}

For $\xi$ with $|\xi|=1$, the system \eqref{ODEh} is of the form
\[ H'' + nH' + Q_0 H + e^{-t} Q_1(\xi) H - e^{-2t}  H = 0, \]
for suitable matrices $Q_0, Q_1$, where only $Q_1$ depends on $\xi$.
After a change of variables $s= e^{-t}$, we obtain
\[ s^2 \frac{d^2}{ds^2} H - (n-1)s\frac{d}{ds} H  + Q_0H + s Q_1(\xi) H -
s^2H = 0. \]
Writing $J = \frac{d}{ds}H$ we have the equivalent first order system
\[ \begin{aligned}
  \frac{d}{ds} H &= J,\\
  \frac{d}{ds} J &= H + (n-1)s^{-1}J - Q_1(\xi) s^{-1}H - Q_0s^{-2}H.
\end{aligned} \]
The leading coefficients are given by the matrix
\[ \begin{bmatrix} 0 & I \\ I & 0 \end{bmatrix}, \]
which has eigenvalues $1,-1$ with multiplicity $(n+1)(n+2)/2$ each.
The system has an irregular singularity of rank 1 as
$s\to\infty$, and so there will be $(n+1)(n+2)/2$ linearly independent
solutions which as $s\to\infty$ have leading order term $s^{-r} e^s$ for suitable
$r$, and $(n+1)(n+2)/2$ solutions which decay like $s^{-r} e^{-s}$.
We are interested in the solutions that blow up as $s\to\infty$, and
for these each component of $H$ has an asymptotic expansion
of the form
\begin{equation}\label{eq:Hasympt}
 H_{ab} \sim s^{-r} e^s(c_{ab} + c^{(1)}_{ab}s^{-1} + c^{(2)}_{ab}
s^{-2} + \ldots).
\end{equation}
If we substitute this asymptotic power series into our
system, then the leading terms are of order $s^{2-r} e^s$,
and these cancel in each equation. The vanishing of the next
order term, $s^{1-r}e^s$ gives rise to a system of linear
equations for the coefficients $\mathbf{c} =c_{ab}$:
\[ -2r \mathbf{c} - (n-1) \mathbf{c}  +
Q_1(\xi)\mathbf{c} = 0, \]
so $\mathbf{c}$ is an eigenvector of the matrix $Q_1(\xi)$, with
eigenvalue $2r + n-1$.

\begin{lem}
  The matrix $Q_1(\xi)$ is diagonalizable, with real eigenvalues.
\end{lem}
\begin{proof}
  This follows from the fact that $Q_1(\xi)$ is self adjoint in a
  suitable basis. More precisely, let us write $A_{jk} = H_{jk}$ for
  $j\ne k$, $A_{j0} = H_{j0}$, $A_{jj} = \frac{1}{\sqrt{2}} H_{jj}$,
  and $A_{00} = \frac{1}{\sqrt{2}} H_{00}$. In this basis we have
  \[ \begin{aligned} (Q_1(\xi)A)_{jk} &= -2\sqrt{-1}(\xi_j A_{k0} + \xi_k A_{j0}),
    \quad \text{ for }j\ne k \\
    (Q_1(\xi)A)_{j0} &= -2\sqrt{-2}\xi_j A_{00} + 2\sqrt{-1} \sum_{k\ne
      j} \xi_k A_{jk} + 2\sqrt{-2} \xi_j A_{jj} \\
    \frac{1}{\sqrt{2}}(Q_1(\xi)A)_{jj} &= -2\sqrt{-2}\xi_j A_{j0} \\
    \frac{1}{\sqrt{2}}(Q_1(\xi)A)_{00} &= 2\sqrt{-2} \sum_j \xi_j
    A_{j0},
    \end{aligned} \]
    so that $Q_1(\xi)$ is self adjoint.
\end{proof}

From this lemma we obtain that there are $(n+1)(n+2)/2$ linearly
independent solutions of our system with asymptotic expansion
\eqref{eq:Hasympt}, where the value of $r$ may depend on the
solution. The type I, II, III solutions that we found in
the previous subsections cannot decay as $s\to \infty$ by the maximum
principle. This can be viewed as an instance of the argument in the
proof of Lemma~\ref{prop:GrahamLee}, or more precisely its linearization around the
hyperbolic metric. To see this note that for any fixed $\xi$ these ODE
solutions define periodic elements $h$ in the kernel of $L_{\mathbf{H}}$
on $\mathbf{R}^n\times (-\infty, \infty)$. Letting $\omega =
\beta_{\mathbf{H}}(h)$ as in the proof of Lemma~\ref{prop:GrahamLee} we
find that $|\omega|$ cannot admit an interior maximum. But $\omega\to
0$ as $t\to\infty$, so $h$ cannot decay as $t\to -\infty$
(i.e. $s\to\infty$) as well. It follows that the type I, II or III
solutions have asymptotics of the form
\eqref{eq:Hasympt} as $s\to\infty$. Translating back to the
$t$-variable, the conclusion is the following.
\begin{prop}
  For any eigenvector $\mathbf{c}=c_{ab}$ of the matrix $Q_1(\xi)$ we
  obtain a solution $H$ of the system $L_\xi(H)=0$. As $t\to -\infty$
  each component $H_{ab}$ has asymptotic expansion
  \[ H_{ab}(t) \sim e^{rt} e^{e^{-t}} (c_{ab} + c^{(1)}_{ab} e^t +
  c^{(2)}_{ab} e^{2t} + \ldots), \]
  while as $t\to \infty$, we have $|H| = O(e^{-nt})$.
\end{prop}

Let us now look at the boundary condition. Substituting $H_{ab}$ into
\eqref{BianH}, the leading terms are
\begin{equation}\label{eq:B2} \begin{aligned}
 B_\xi(H)_i &=  \Big[- c_{i0} -
 \sqrt{-1}\xi_j c_{ij} + \frac{1}{2}
 \sqrt{-1}\xi_i c_{pp}+
\frac{1}{2} \sqrt{-1} \xi_i c_{00} \Big] e^{(r-1)t}e^{e^{-t}}  + O(e^{rt}e^{e^{-t}})\\
  B_\xi(H)_0 &=\Big[ -\frac{1}{2}c_{00} +
  \frac{1}{2}c_{pp} - \sqrt{-1} \xi_i c_{i0}\Big] e^{(r-1)t}e^{e^{-t}}
  + O(e^{rt}e^{e^{-t}}),
\end{aligned} \end{equation}
and more precisely $B_\xi(H)$ has an asymptotic expansion in powers of
$e^t$. Note that the leading coefficients do not depend on $r$. Let us write
\[ B_\xi(H) \sim e^{(r-1)t}e^{e^{-t}} \left( R(\xi)\mathbf{c} +
  R^{(1)}(\xi,r)\mathbf{c} e^t + \ldots \right), \]
where $R(\xi)$ is independent of $r$.
We have the following
\begin{lem}\label{lem:Rinvertible}
  The matrix $R(\xi)$ has a right inverse for all $\xi$ with $|\xi|=1$.
\end{lem}
\begin{proof}
  Since our problem is rotationally invariant in $\mathbf{R}^n$, it is
  enough to check this for a single unit vector $\xi$, for instance
  $\xi=(1,0,0,\ldots,0)$, in which case it is straight forward.
\end{proof}

Let us fix an eigenvector $\mathbf{c}$ of $Q_1(\xi)$, and define
\[ \tilde{H}(t) = H(t - T) e^{rT}e^{-e^T}. \]
Then, using Lemma~\ref{lem:translate}, we have
\[ \begin{aligned}
  L_{e^T\xi}(\tilde{H}) &= 0, \\
  \tilde{H}_{ab}(0) &\sim c_{ab} + c^{(1)}_{ab}e^{-T} + \ldots, \\
  B_{e^T\xi}(\tilde{H})|_{t=0} &\sim e^T R(\xi)\mathbf{c} + R^{(1)}(\xi,r)
  \mathbf{c}  + \ldots.
\end{aligned} \]
Writing $\zeta = e^T\xi$, and recalling that $|\xi|=1$, we have
\[ \begin{aligned}
  L_{\zeta}(\tilde{H}) &= 0, \\
  \tilde{H}_{ab}(0) &\sim c_{ab} + c^{(1)}_{ab} |\zeta|^{-1} + \ldots, \\
  B_{\zeta}(\tilde{H})|_{t=0} &\sim |\zeta| R(|\zeta|^{-1}\zeta)\mathbf{c} + R^{(1)}(|\zeta|^{-1}\zeta,r)
  \mathbf{c} + \ldots.
\end{aligned} \]

From this and Lemma~\ref{lem:Rinvertible}
it follows that as long as $|\zeta|$ is sufficiently large,
say $|\zeta| > \kappa^{-1}$, we can take suitable linear combinations of
our solutions $\tilde{H}$ (for different eigenvectors $\mathbf{c}$)
with coefficients that have an asymptotic expansion in powers of
$|\zeta|^{-1}$, and obtain $H^a$ satisfying
$L_\zeta(H^a)=0$,  $B_\zeta(H^a)|_{t=0} = \mathbf{e}_a$, and
\[ H^a(0) \sim  \Psi_a^{(-1)}(\zeta) +
\Psi^{(-2)}_a(\zeta) + \Psi^{(-3)}_a(\zeta) + \ldots, \]
where for each $i$, $\Psi_a^{(i)}$ is homogeneous of degree $i$, and
smooth on the unit sphere.
In addition we have $H^a(t) = O(e^{-nt})$ as $t\to\infty$. More
precisely we have the following estimate.
\begin{prop}\label{prop:Hadecay}
  For $|\zeta| > \kappa^{-1}$ the solutions $H^a$ satisfy
  \[ |\partial_\zeta^i H^a(t)| \leq C_i |\zeta|^{-1-i} e^{-(n-1)t}, \]
  for all $i$, and $t\geq 0$, with suitable constants $C_i$.
\end{prop}
\begin{proof}
  Let $|\xi|=1$,  fix an eigenvector $\mathbf{c}$ of $Q_1(\xi)$ as
  above, and let $H(t)$ be the corresponding solution of
  $L_\xi(H)=0$. From the asymptotic behavior of $H$ as $t\to -\infty$
  we have that for suitable $c,C > 0$
  \[ (\log |H|)' < -c e^{-t} < -(n-1) \]
  for $t < -C$, while the behavior as $t\to\infty$ implies that
  \[ (\log |H|)' < -(n-1) \]
  for $t > C$. It follows from this that for any $s\in \mathbf{R}$ and
  $t\geq 0$, we have
  \[ \log |H|(s+t) < \log |H|(s) + C - (n-1)t,\]
  i.e.
  \begin{equation}\label{eq:Hst}
     |H|(s+t) \leq C e^{-(n-1)t} |H|(s)
  \end{equation}
  for a different constant $C$. Since the derivatives $\partial_t^i H$ have
  analogous asymptotics to $H$, they also satisfy estimates of
  the form \eqref{eq:Hst}.

  For large $T$, let
  \[ \tilde{H}_T(t) = H(t-T)e^{rT}e^{-e^T} \]
  as above. By the asymptotics of $H$, we have that $e^{rt}e^{-e^T}
  |H|(-T)$ is bounded for large $T$, and so using \eqref{eq:Hst} we
  have, for $t\geq 0$, that
  \[ |\tilde{H}_T(t)| \leq Ce^{-(n-1)t}, \]
  with $C$ independent of $T$. Computing a derivative
  \[ e^{-T} \partial_T \tilde{H}_T(t) = \Big[ -e^{-T} H'(t-T) - H(t-T)
  + r e^{-T}
  H(t-T)\Big] e^{rt}e^{-e^T}, \]
  and so using the analogous estimate to \eqref{eq:Hst} for $H'$
  together with a bound on $e^{(r-1)T}e^{-e^T} |H'|(-T)$ for large
  $T$, we obtain
  \[ |e^{-T} \partial_T \tilde{H}_T(t)| \leq Ce^{-(n-1)t}, \]
  for $t \geq 0$.
  We can bound further derivatives $(e^{-T}\partial_T)^i \tilde{H}_T$
  in a similar way.

  Using the substitution $T = \log|\zeta|$, this implies that
  \[ |\partial_\zeta^i \tilde{H}_{\log|\zeta|}(t)| \leq
  Ce^{-(n-1)t} \]
  for $t\geq 0$. The solutions $H^a$ are obtained by taking linear
  combinations of such $\tilde{H}$, with coefficients that are of
  order $|\zeta|^{-1}$, and have an asymptotic expansion in powers of
  $|\zeta|^{-1}$. The required estimates follow from this.
\end{proof}

\subsection{Prescribing the Bianchi operator for all $\xi$}
We have seen in section~\ref{subsec:smallxi}
 that for sufficiently small $\xi$ we can find
solutions $\tilde{H}^a$ of \eqref{ODEh}, such that
$B_\xi(\tilde{H}^a)|_{t=0} = \mathbf{e}_a$. Applying
Lemma~\ref{lem:translate} this means that if we fix $\xi$ with
$|\xi|=1$, then we have solutions $H^a$ of $L_\xi(H^a)=0$ with
$B_\xi(H^a)|_{t=T} = \mathbf{e}_a$ for some large $T$. A crucial
result is the following.

\begin{prop} The vectors $B_\xi(H^a)(t) \in \mathbf{C}^{n+1}$ are
  linearly independent for all $t\in \mathbf{R}$.
\end{prop}
\begin{proof}
 This follows from the maximum principle, analogously to
 Lemma~\ref{prop:GrahamLee}. Indeed, if there was a value of $t$ at
 which the vectors were not linearly independent, then we could form a
 linear combination and take the inverse Fourier transform to obtain
 a periodic element $h$ in the kernel of $L_{g_{\mathbf{H}}}$ for
 which $\omega = \beta_{g_{\mathbf{H}}}(h)$ vanishes at some value of
 $t$. This contradicts that $|\omega|$ cannot admit an interior
 maximum.
\end{proof}

Applying Lemma~\ref{lem:translate} again, it follows that for all
$\xi$ we can find suitable solutions $H^a$ of $L_\xi(H^a)=0$,
satisfying $B_\xi(H^a)|_{t=0}= \mathbf{e}_a$. In the previous two
subsections we have constructed special collections of such $H^a$ for
sufficiently small, and for sufficiently large $|\xi|$
respectively. Combining these with suitable cutoff functions, we
obtain the following.

\begin{prop}\label{prop:Fouriersolutions}
For all $\xi\ne 0$ we have solutions $H^a(\xi,t)$ of
$L_\xi(H^a)=0$, $B_\xi(H^a)|_{t=0} = \mathbf{e}_a$, which depend
smoothly on $\xi$ such that in
addition we have
\begin{enumerate}
\item For small $\xi$
  \begin{equation}\label{eq:HaPhi}
   H^a(\xi, 0) = |\xi|^{-2} \Phi^a(\xi)
\end{equation}  for smooth $\Phi^a$ with $\Phi^a(0)=0$,
\item For large $\xi$ we have an asymptotic expansion
  \begin{equation}\label{eq:Haasympt}
 H^a(\xi, 0)\sim \Psi_a^{(-1)}(\xi) +
\Psi^{(-2)}_a(\xi) + \Psi^{(-3)}_a(\xi) + \ldots,
\end{equation}
where each $\Psi_a^{(i)}$ is homogeneous of degree $i$, and smooth
on the unit sphere.
\item For $t\geq 0$ and all $\xi\ne 0$ we have
  \begin{equation} \label{eq:tdecay}
 |\partial_\xi^i H^a(\xi, t)| \leq C_i |\xi|^{-1-i}
  e^{-(n-1)t},
  \end{equation}
  for constants $C_i$.
\end{enumerate}
\end{prop}

We can now state the main result of this subsection.
\begin{prop}\label{prop:Hinverse2}
  Suppose
  that $\eta \in T^*\mathbf{H}^{n+1}|_{t=0}$ is a one-form, satisfying
  the following estimates:
  \begin{enumerate}
    \item $\Vert \eta\Vert_{C^{1,\alpha}} \leq C$,
    \item $\Vert \eta\Vert_{C^{1,\alpha}(B_R \setminus B_{R-1})} \leq
        CR^{-2n+1/2}$ for $R > 1$,
  \end{enumerate}
  and in addition for all $i=1, \ldots, n$, and each component $\eta_a$ we have
  \begin{equation}\label{eq:inteta}
 \int_{\mathbf{R}^n} \eta_a\, dx = \int_{\mathbf{R}^n} x_i
  \eta_a\, dx = 0.
\end{equation}
  Then there exists a symmetric two tensor $h\in
  C^{2,\alpha}_\delta(\mathbf{H}^{n+1}_+)$ satisfying $Lh = 0$ in
  $\mathbf{H}^{n+1}_+$, such that $h$ has the boundary condition
  $\eta$, i.e. $\beta(h)|_{t=0} = \eta$ and
  in addition $h$ satisfies the following decay estimate, for any
  $\delta\in (0,1)$:
  \begin{equation}\label{eq:hdecay}
      \Vert h\Vert_{C^{2,\alpha}_1(A_{R-1,R})} <
      C'(1+R)^{-n-1+\delta}
   \end{equation}
    for all $R > 0$,  where $C'$ depends on the constant $C$ and on $\delta$.
\end{prop}
\begin{proof}
  We use the solutions $H^a$ of $L_\xi(H^a)=0$ from
  Proposition~\ref{prop:Fouriersolutions} to define
  \[ H(\xi, t) = \sum_a \widehat{\eta_a}(\xi) H^a(\xi,
  t). \]
  Then the inverse Fourier transform $h(x,t)$ of $H$ will satisfy $Lh
  = 0$, and by construction $\beta(h)|_{t=0} = \eta$ will hold. What remains is
  to verify that $h$ satisfies the required estimates. We will first
  focus on the relevant estimates at $t=0$.

  Let us define a cutoff function $\rho$ such that $\rho(s)=1$ for $s
  < 1/2$, and $\rho(s)=0$ for $s > 1$.
  Let us write $h=h_1 + h_2$ where $h_1$ is the inverse
  Fourier transform of $H_1 = \rho(|\xi|) H$, and $h_2$ is the
  inverse transform of $H_2 = (1-\rho(|\xi|))H$. I.e. we collect the
  small Fourier modes in $h_1$, and the large ones in $h_2$. We prove
  the required estimates for $h_1,h_2$ separately.

  We have $h_1 = \sum_a h_1^a$, where
  \[ h^a_1(x,0) = \mathcal{F}^{-1}\Big[ \rho(|\xi|)
  \widehat{\eta_a}(\xi) |\xi|^{-2} \Phi^a(\xi)\Big], \]
  in terms of the formula \eqref{eq:HaPhi} for $H^a$, and the
  inverse Fourier transform $\mathcal{F}^{-1}$.
  In terms of convolutions we have
  \[ h^a_{1} (x,0) = \eta_a \ast \mathcal{F}^{-1}\Big[ \rho(|\xi|)
  \Phi^a(\xi)\Big] \ast
  \mathcal{F}^{-1}(|\xi|^{-2}).  \]
  Since $\rho(|\xi|)$ is smooth and compactly supported, the inverse
  Fourier transform of $\rho(|\xi|) \Phi^a(\xi,t)$ is a Schwarz
  function. It follows that the function
  \[ N(x) = \eta_a \ast \mathcal{F}^{-1}\Big[\rho(|\xi|)
  \Phi^a(\xi)\Big] \]
  satisfies the same decay estimates as $\eta_a$, but for all
  derivatives rather than just the $C^{1,\alpha}$ norm. In addition we
  have $\Phi^a(0)=0$, and the assumption \eqref{eq:inteta}
  implies $\hat{\eta}(0)= \partial_{\xi_i}
  \hat{\eta}(0)= 0$ for all $i$. For $N$ this implies
  \[ \int_{\mathbf{R}^n} N(x)\, dx = \int_{\mathbf{R}^n} x_i
  N(x)\,dx = \int_{\mathbf{R}^n} x_i x_j N(x)\,dx =0 \]
  for all $i, j$.

  At the same time in the sense of distributions we have
  \[ \mathcal{F}^{-1}\left(|\xi|^{-2}\right) = c |x|^{2-n}, \]
  for a dimensional constant $c$, so
  \[ h^a_{1}(x,t) = c\,N(x)\ast |x|^{2-n}. \]
  The required decay estimate for $h^a_{1}$ follows from this.

  Let us now consider $h_2$. Then $h_2$ is a sum of terms $h_2^a$, where
  \[ \begin{aligned}
       h_2^a(x,0) &= \mathcal{F}^{-1}\Big[ (1-\rho(|\xi|)) \hat{\eta}_a(\xi)
  H^a(\xi)\Big] \\
       &= \eta_a \ast \mathcal{F}^{-1}\Big[ (1-\rho(|\xi|)) H^a(\xi)
       \Big].
   \end{aligned} \]
  Using the asymptotic expansion \eqref{eq:Haasympt}
   for $H^a(\xi)$, we have
   \[ H^a(\xi) \sim \Psi_a^{(-1)}(\xi) + \ldots + \Psi_a^{(-n-2)}(\xi) +
   \Theta_a(\xi), \]
  where for large $\xi$ we have $\nabla_\xi^k \Theta_a(\xi) =
  O(|\xi|^{-n-3-k})$. It follows that
  \[ K_{\Theta_a}(x) = \mathcal{F}^{-1}\Big[ (1-\rho(|\xi|)) \Theta_a(\xi) \Big] \]
  is smooth on $\mathbf{R}^n$, and  $K_{\Theta_a}, \nabla_x K_{\Theta_a},
  \nabla^2_x K_{\Theta_a}$ decay exponentially fast. In particular
  $\eta_a\ast K_{\Theta_a}$ satisfies the required estimates.

  Let us write
  \[ K^{(-i)}_a = \mathcal{F}^{-1}\Big[ (1-\rho(|\xi|))
  \Psi_a^{(-i)}\Big]. \]
  Then the distribution $\nabla_x^i K^{(-i)}_a$ is the Fourier transform of a
  function which for large $\xi$ is homogeneous of degree zero. The
  decay of the derivatives of $\Psi_a^{(-i)}$ implies that
  $K^{(-i)}_a$ has singular support at the origin, and all of its
  derivatives decay
  exponentially fast away from the origin. It
  follows from these properties (as in Gilbarg-Trudinger, Section 4.3 for the
  Poisson equation) that for each $i$,
  \[   \eta_a \ast \nabla_x^i K^{(-i)}_a \]
  decays in $C^{1,\alpha}$ (or in any other H\"older space) at the
  same rate as $\eta_a$. Since $i \geq 1$, we obtain the required
  $C^{2,\alpha}$ estimates for $h_2^a(x,0)$.

  We now consider $h(x,t)$ for $t\geq 0$. Our goal is to show that
  $e^{t}|h(x,t)| \leq C'(1+|x|)^{-n-1+\delta}$, since then Schauder estimates together
  with our estimate for the boundary values of $h$ imply the required
  $C^{2,\alpha}$ estimates. As above, for each $t$, $h(x,t)$ is obtained as a convolution
  of components of $\eta$ with the Fourier transforms of the solutions
  $H^a(\xi,t)$ of Proposition~\ref{prop:Fouriersolutions}. The
  property \eqref{eq:tdecay} together with
  Lemma~\ref{lem:Fourierdecay} and Lemma~\ref{lem:convolutiondecay}
  below implies the result.
\end{proof}

\begin{lem}\label{lem:Fourierdecay}
  Suppose that $f : \mathbf{R}^n\to\mathbf{R}$ satisfies the estimates
  \begin{equation}\label{eq:dxif} |\partial_x^i f(x)| \leq C_i
    |x|^{-1-i}
  \end{equation}
  for $x\ne 0$. For any $\delta\in (0,1)$, the Fourier transform
  $\hat{f}$ of $f$ in the sense of distributions then satisfies
  \begin{equation}\label{eq:dxif2}
 |\partial_\xi^j \hat{f}(\xi)| \leq \begin{cases} C_j'
    |\xi|^{1-n-j-\delta}, &\text{ for } |\xi| < 2 \\
    C_j' |\xi|^{1-n-j+\delta}, &\text{ for } |\xi| > 1.
   \end{cases}
  \end{equation}
\end{lem}
\begin{proof}
  We will first prove the $j=0$ case of the required inequality,
  i.e. we prove that assuming the estimate \eqref{eq:dxif}, we have
  \[ |\hat{f}(\xi)| \leq \begin{cases} C_j'
    |\xi|^{1-n-\delta}, &\text{ for } |\xi| < 2 \\
    C_j' |\xi|^{1-n+\delta}, &\text{ for } |\xi| > 1.
   \end{cases} \]
  Let us write $f=f_1 + f_2$, where $f_1$ is supported in $B_2(0)$,
  and $f_2$ is supported on $\mathbf{R}^n\setminus B_1(0)$.

  Consider $f_1$ first. Let us write
  \[ \frac{n-1-\delta}{2} = k + s, \]
  where $k\in \mathbf{Z}$ and $s\in (0,1)$. The estimates \eqref{eq:dxif} imply that
  \[
   |\Delta_x^{k+s} f_1(x)| \leq \begin{cases} C |x|^{-n+\delta},
  &\text{ for } 0 < |x| < 2, \\
   C |x|^{-n-2s}\, &\text{ for } |x| > 1, \end{cases}
  \]
  where we are using the fractional Laplacian $\Delta^s$. To see this,
  note that from our assumptions we have
  \[ |\partial_x^i \Delta_x^k f_1(x)| \leq C|x|^{-1-k-i}, \]
  and the required estimate then follows from the integral formula
  \[ \Delta_x^{k+s} f_1(x) = c_{n,s}
  \int_{\mathbf{R}^n} \frac{u(x+y) + u(x-y) - 2u(x)}{|y|^{n+2s}}\,
  dy, \]
  where $u(x) = \Delta_x^k f_1(x)$. In particular we find that
  $\Delta_x^{k+s} f_1 \in L^1$, and so on the Fourier transform side
  we obtain that $|\xi|^{2(k+s)} \hat{f}_1$ is bounded, i.e.
  \[  | \hat{f_1}(\xi) | < C|\xi|^{1-n+\delta}, \]
  for all $\xi$. At the same time the fact that $f_1$ is compactly
  supported implies that $\hat{f}_1$ is actually smooth and in
  particular it is bounded near $\xi=0$.

  We can deal with $f_2$ in a similar way, letting
  \[ \frac{n-1+\delta}{2} = k + s\]
  this time. Since $\partial_x^N f_2 \in L^1$ for all $N > n-1$, it
  follows that $\hat{f}_2$ decays at infinity faster than any
  polynomial, while a similar argument to the above, using the fractional
  Laplacian, shows that
  \[ |\hat{f}_2(\xi) | < C|\xi|^{1-n-\delta} \]
  for $0 < |\xi| < 1$, say. Combining these estimates for $\hat{f}_1,
  \hat{f}_2$, we obtain the required bound for $\hat{f}$.

  Given the estimate \eqref{eq:dxif2} for $j=0$, we can obtain the
  general case if we replace $f$ by $P_j(x)f(x)$ for degree $j$
  monomials $P_j$.
\end{proof}

\begin{lem}\label{lem:convolutiondecay}
  Suppose that $f:\mathbf{R}^n\to\mathbf{R}$ satisfies  $|f(x)|\leq
  C(1 + |x|)^{-2n+1/2}$, and
  \[ \int_{\mathbf{R}^n} f(x)\,dx = \int_{\mathbf{R}^n} x_i f(x)\,dx =
  0\]
  for each $i$. Let $K:\mathbf{R}^n\to\mathbf{R}$ be such that
  for some $\delta\in (0,1)$,
  \[ |\partial_x^i K(x)| \leq \begin{cases} C_i |x|^{1-n-\delta},
    &\text{ for }|x| < 2\\
   C_i |x|^{1-n+\delta}, &\text{ for }|x| > 1. \end{cases} \]
  Then the convolution $g = f\ast K$ satisfies $|g(x)| \leq
  C'(1+|x|)^{-1-n+\delta}$.
\end{lem}
\begin{proof}
  We can expand $K(x-y)$ in a Taylor series around $y=0$, and the
  series will converge on the region $|y| < |x|/2$, say:
  \[ K(x-y) = K(x) - y_i\partial_i K(x) + O(|y|^2
    |x|^{-1-n+\delta}). \]
  We then have
  \[\begin{aligned}  g(x) &= \int_{|y| < |x|/2} f(y) K(x-y)\,dy + \int_{|x|/2 \leq |y|
  \leq 2|x|} f(y) K(x-y)\,dy \\ &\quad + \int_{|y| > 2|x|} f(y)
K(x-y)\,dy.
\end{aligned} \]
  In estimating the first integral we use the Taylor expansion of
  $K(x-y)$, while the other two integrals can be estimated directly.
\end{proof}


\subsection{The proof of Theorem~\ref{thm:Hinverse}}
In this section we will give the proof of
Theorem~\ref{thm:Hinverse}. The first step is to
solve the inhomogeneous problem in (\ref{h0}):

\begin{prop}\label{inhomogprop}
  Suppose that $u \in C^{0,\alpha}_1$
  is a symmetric two-tensor on $\mathbf{H}^{n+1}_+$ supported in $B_1 \times
  [0,\infty)$. Then there exists a symmetric two-tensor $h_0$ on
  $\mathbf{H}^{n+1}_+$ satisfying $Lh_0 = u$, and $h_0$ satisfies the
  estimate
  \[ \Vert h_0\Vert_{C^{2,\alpha}_1(A_{R-1,R})} \leq C (1+R)^{-2n+1/2}
    \Vert u\Vert_{C^{0,\alpha}_1}, \]
  for all $R > 0$.
\end{prop}
\begin{proof}
  First by reflecting $u$ across the boundary of $\mathbf{H}^{n+1}_+$,
  and multiplying by a cutoff function, we extend $u$ to a tensor
  $\tilde{u}$ on all of hyperbolic space $\mathbf{H}^{n+1}$, with
  \[\Vert \tilde{u}\Vert_{C^{0,\alpha}_1(\mathbf{H}^{n+1})} \leq C
  \Vert u\Vert_{C^{0,\alpha}_1(\mathbf{H}^{n+1}_+)}.\]
  We can then apply Theorem~\ref{LeeThm} to obtain the required tensor
  $h_0$ on all of $\mathbf{H}^{n+1}$, and we simply restrict it to
  $\mathbf{H}^{n+1}_+$. The required decay of $h_0$ in the
  $x_i$-directions follows from the decay result \cite[Proposition
  5.2]{LeeMemoirs}.
\end{proof}

The next step in the proof of Theorem~\ref{thm:Hinverse}  is to let
$\eta = \beta_{\mathbf{H}}(h_0)|_{t=0}$, and try
using Proposition~\ref{prop:Hinverse2} to find $h_1$ such that
$Lh_1=0$, and $\beta_{\mathbf{H}}(h_1)|_{t=0} = \eta$. For this we
need to check the integral conditions \eqref{eq:inteta}, which are
equivalent to $\hat{\eta}(0) = \partial_i\hat{\eta}(0)=0$. This is where
the condition $I(u)=0$ enters, but we will need to further
adjust $h_0$ before these conditions hold. We
first have the following.

\begin{prop}
  Suppose that $u\in C^{0,\alpha}_1$ satisfies $I(u) =0$, and that
  $h_0$ satisfies $Lh_0=u$. Let $\eta =
  \beta_{\mathbf{H}}(h_0)|_{t=0}$. Then for small $\xi$ the components of the
  Fourier transform $\hat{\eta}$ satisfy
  \begin{equation}\label{eq:bian1}
        \hat{\eta}_i(\xi) = \xi_j A_{ij} + O(|\xi|^2),
  \end{equation}
  for a symmetric matrix $A_{ij}$.
\end{prop}
\begin{proof}
  We need to show that $\hat{\eta}_i(0)=0$, and that the
  skew-symmetric part of the first derivative of $\hat{\eta}_i$
  vanishes at the origin, i.e.
  \[ \partial_{\xi_j} \hat{\eta}_i(0) - \partial_{\xi_i}
  \hat{\eta}_j(0)=0. \]

  Let us denote by $H_0(\xi, t)$ the Fourier transform of
  $h_0$ with additional exponential factors as before in
  Equation~\eqref{FT}. Similarly $U(\xi, t)$ is the Fourier transform
  of $u$ with additional exponential factors.
  The equation $Lh_0 = u$ then implies that $L_\xi
  H_0 = U$, where $L_\xi$ is the operator given by the left hand side
  of \eqref{ODEh}. In particular, the components $H_{0,j0}(0,t)$
  for $\xi=0$ satisfy the ODEs
  \[ H_{0,j0}'' + nH_{0,j0}' - (n+1)H_{0,j0} = U_{j0}. \]
  The condition $I(u)=0$ says that for all $x$ we have
  \begin{equation}\label{eq:intu1}
   \int_0^\infty u_{j0}(x,t) e^{-2t}\,dt =0,
  \end{equation}
  and so taking the Fourier transform, and letting $\xi=0$ we get
  \[ \int_0^\infty U_{j0}(0,t) e^{-t}\, dt =0, \]
  recalling that $U_{j0} = e^{-t}\hat{u_{j0}}$.
  Applying Lemma~\ref{lem:ODEintegral}, we find that for each $j$,
  \[ H_{0,j0}'(0,0) + (n+1)H_{0,j0}(0,0) =0. \]
  Using the formula \eqref{BianH} for the Fourier transform of the
  Bianchi operator, we then have
  \[ \hat{\eta}_j(0) = H_{0,j0}'(0,0) + (n+1)H_{0,j0}(0,0) = 0, \]
  as required.

  We next look at the first derivative of $\hat{\eta}$, and for this
  we differentiate the equation $L_\xi H_0 = U$ with respect to
  $\xi$. We only need certain components of the derivative, so let us
  define
\[ S_{ij}(t) = \partial_{\xi_i} H_{0,j0} - \partial_{\xi_j}
H_{0,i0} \Big|_{\xi = 0}. \]
Differentiating the equation $L_\xi H_0 = U$ with respect to $\xi$ and
then setting $\xi=0$, we obtain
\[  S_{ij}'' + nS_{ij}' - (n+1)S_{ij} = g_{ij}(t), \]
where
\[ g_{ij}(t) = \partial_{\xi_i}U_{j0} - \partial_{\xi_j}U_{i0}\Big|_{\xi=0}. \]
From the properties of the Fourier transform we have
\[ g_{ij}(t) = -\sqrt{-1} \int_{\mathbf{R}^n} x_j e^{-t}u_{i0}(x,t) - x_i
e^{-t}u_{j0}(x,t)\,dx , \]
and so \eqref{eq:intu1} for all $x$ implies
\[ \int_0^\infty g_{ij}(t) e^{-t}\,dt = 0. \]
Just as above, Lemma~\ref{lem:ODEintegral}
then implies that $S_{ij}'(0) + (n+1) S_{ij}(0) = 0$.

At the same time \eqref{BianH} implies that
\[ \begin{aligned} \partial_{\xi_j} \hat{\eta}_i - \partial_{\xi_i}
  \hat{\eta}_j\Big|_{\xi=0} &= S_{ij}'(0) + (n+1) S_{ij}(0) = 0,
\end{aligned} \]
which is what we wanted to prove.
\end{proof}

We used the following in the previous argument.
\begin{lem}\label{lem:ODEintegral}
 Suppose that $f: [0,\infty)\to \mathbf{R}$ is a decaying solution of
\[ f'' + nf' - (n+1)f = g. \]
  Then $f'(0) + (n+1)f(0) = 0$ if and only if
\[ \int_0^\infty g(s) e^{-s}\,ds = 0. \]
\end{lem}
\begin{proof}
  The solutions of the homogeneous equation are $e^t,
  e^{-(n+1)t}$. Note that the decaying homogeneous
  solution $\phi(t) = e^{-(n+1)t}$
  satisfies $\phi'(0) + (n+1)\phi(0)=0$, and so it is enough to check
  the statement of the lemma for one particular solution.

  A decaying fundamental solution of the ODE is
  \[ \Gamma(t) = \begin{cases}
    ae^t\, \text{ if } t<0 \\
    ae^{-(n+1)t}, \text{ if } t > 0,
  \end{cases} \]
  for a suitable constant $a$, and so a decaying solution of the ODE
  is
  \[ f(t) = \int_{-\infty}^\infty g(s) \Gamma(t-s)\, ds. \]
  It follows that
  \[  \begin{aligned}
     f'(0) + (n+1)f(0) & = \int_{-\infty}^\infty g(s) \big[
     \Gamma'(-s) + (n+1)\Gamma(-s)\big]\, ds \\
    &= \int_0^\infty g(s) \big[ a e^{-s} + (n+1)a e^{-s}\big]\, ds \\
    &= a(n+2) \int_0^\infty g(s) e^{-s}\,ds.
  \end{aligned} \]
   The result follows.
\end{proof}

We are now ready to complete the proof of
Theorem~\ref{thm:Hinverse}. Consider again $\eta =
\beta_{\mathbf{H}}(h_0)|_{t=0}$ for the $h_0$ given by
Proposition~\ref{inhomogprop}. Using solutions of \eqref{ODEh}
for small $\xi$ satisfying \eqref{eq:solnsmallxi} we can find a
solution $\tilde{H}(\xi, t)$ of $L_\xi(\tilde{H}) =0$, vanishing for
$|\xi| > 1$, depending smoothly on $\xi$, and such that
\[ \hat{\eta} - B_\xi(\tilde{H})|_{t=0} = O(|\xi|^2). \]
The inverse Fourier transform $\tilde{h}$ of $\tilde{H}$ decays
exponentially fast (it is in the Schwarz space).
We can then apply Proposition~\ref{prop:Hinverse2} to find $h_1$
satisfying $Lh_1 = 0$, and
\[ \beta_{\mathbf{H}}(h_1)|_{t=0} = \eta -
\beta_{\mathbf{H}}(\tilde{h})|_{t=0}, \]
as well as the decay estimates \eqref{eq:hdecay}. We finally let
\[ h = h_0 - \tilde{h} - h_1. \]
This satisfies $Lh = Lh_0 = u$, the boundary condition
$\beta_{\mathbf{H}}(h) =0$, and the required decay estimates.

We also have the following improvement over
Proposition~\ref{prop:Hinverse2} when the only nonzero component of
$\eta$ is $\eta_0$.
\begin{prop}\label{prop:Hinverse3}
  Suppose that $f: \mathbf{R}^n\to\mathbf{R}$ is a $C^{1,\alpha}$ function supported
  in the unit ball $B_1$. There exists a symmetric two tensor $h\in
  C^{2,\alpha}_1(\mathbf{H}^{n+1}_+)$ satisfying
  \begin{enumerate}
    \item $Lh = 0$,
      \item $\beta(h)_i|_{t=0} = 0$ for all $i$ and $\beta(h)_0|_{t=0}
        = f$,
      \item $h$ satisfies the decay estimate
        \[ \Vert h\Vert_{C^{2,\alpha}_1}(A_{R-1,R}) <
        C'_k(1+R)^{-k} \Vert f\Vert_{C^{1,\alpha}}, \]
        for any $k>0$, and $C'_k$ depending on $k$.
   \end{enumerate}
\end{prop}
\begin{proof}
  The solution $h$ is constructed using the Fourier transform just
  like before, but for small $\xi$ only the solutions $H^a$ do not
  have a $|\xi|^{-1}$ singularity this time, as can be seen in
  \eqref{eq:solnsmallxi}. This translates to better decay properties
  of $h$ without the need for a condition like \eqref{eq:inteta}.
\end{proof}

\section{The linearized problem on $[0,\infty) \times M$}  \label{SecLinOp}

In this section we use Theorem~\ref{thm:Hinverse} to invert
the linearized operator on $M\times [0,\infty)$, at first modulo a
finite dimensional space. In this and subsequent sections we will need
to do some local calculations with respect to the warped product
metric
\[ g_{\epsilon} = dt^2 + e^{2t} \epsilon^{-2}  g_M.\]
In particular, $\{ x^i \}$ will denote local coordinates on $M$, and $\{ x^1, \dots, x^n, x^0 = t \}$ the corresponding coordinate system
 on $M \times [0,\infty)$.  We will use $a,b,c,\ldots$ for indices ranging
 from $0$ to $n$, and $i,j,k,\ldots$ for those ranging from 1 to $n$, as
 before. We will also use the obvious identifications between vector
 fields on $M$ and $[0,\infty)$ and their lifts to vector fields on
 the product manifold,
 usually without comment.  Recall the improved approximate
 solution
\[ g_\epsilon' = g_\epsilon + k^{(2)} + e^{-2t}\epsilon^2 k^{(4)}, \]
where $k^{(2)}, k^{(4)}$ are fixed tensors on $M$ expressed in terms of
$g_M$. Since $g_\epsilon$ is uniformly equivalent to $g_\epsilon'$, we
can use either of them to measure norms.

We now compute the Bianchi operator and the variation of the Ricci
curvature with respect to the metric $g_\epsilon$. The nonzero
Christoffel symbols are given by
  \[ \begin{aligned}
    \Gamma^i_{jk} &= \Gamma^i_{M, jk} \\
    \Gamma^0_{jk} &= -\epsilon^{-2} e^{2t} g_{M, jk} = -g_{jk} \\
    \Gamma^i_{0k} &= \Gamma^i_{k0}  =\delta^i_k,
  \end{aligned} \]
where $\Gamma^i_{M, jk}$ denote the Christoffel symbols of $g_M$.

The general formula for the variation of the Ricci curvature is
\begin{align} \label{DRicci}
 D\mathrm{Ric}_{g_\epsilon} (h) = -\frac{1}{2}\nabla^a\nabla_a h +
\delta_{g_\epsilon}^*\beta_{g_\epsilon}(h) + \mathcal{R}(h),
\end{align}
where
\[ \beta_{g_\epsilon}(h)_a = \nabla^b h_{ab} - \frac{1}{2}\nabla_a(
g_\epsilon^{bc}h_{bc}), \]
\[ \delta_{g_\epsilon}^*(\omega)_{ab} = \frac{1}{2}(\nabla_a \omega_b +
\nabla_b\omega_a), \]
and
\[ \mathcal{R}(h)_{cd} = -R^{a\,\,b}_{\,\,c\,\,d} h_{ab}
+\frac{1}{2}(R^a_c h_{ad} + R^{a_d}h_{ac}), \]
in terms of the curvature tensor of $g_\epsilon$.

We are particularly interested in the $j0$-component of the variation
of the Ricci curvature. For this we have the following formulas:
\[ \begin{aligned}
 \nabla^a\nabla_a h_{j0} &= \epsilon^2e^{-2t} g_M^{ik} \nabla_i^M\nabla_k^M
h_{j0} + \partial_0^2 h_{j0} + (n-2)\partial_0 h_{j0} - (4n+2) h_{j0}
\\
  &\quad + 2\partial_j h_{00} - 2\epsilon^2 e^{-2t} g_M^{ik} \nabla^M_k h_{ij},
\end{aligned} \]

\[\begin{aligned}
 2\delta_{g_\epsilon}^*\beta_{g_\epsilon}(h)_{j0} &= \epsilon^2e^{-2t} g_M^{ik} \nabla^M_j\nabla^M_k
 h_{0i} + \partial_0^2h_{j0} + (n-2)\partial_0 h_{j0} - 2n h_{j0} \\
 &\quad - \epsilon^2e^{-2t} \partial_j \partial_0(g_M^{ik} h_{ik}) +
 \epsilon^2e^{-2t} \partial_0( g_M^{ik} \nabla^M_k h_{ji}) \\
&\quad + 2\epsilon^2e^{-2t} \partial_j
 (g_M^{ik}h_{ik}) -4 \epsilon^2e^{-2t} g_M^{ik}\nabla^M_k h_{ji} +
 (n+1)\partial_j h_{00}.
\end{aligned} \]

The curvature of $g$ satisfies
\[ \begin{aligned}
  R^0_{\,\,jk0} &= \epsilon^{-2}e^{2t}g_{M,jk} \\
  R_{00} &= -n \\
  R_{jk} &= R^M_{jk} - ng_{jk},
\end{aligned} \]
where $R^M$ is the Ricci curvature of $g_M$.

\[\begin{aligned}
  \mathcal{R}(h)_{j0} = -(1+n)h_{j0} + \frac{1}{2} \epsilon^2 e^{-2t} g_M^{ik}
  R^M_{kj} h_{i0}.
\end{aligned} \]

Combining all of these we obtain
\begin{equation}\label{eq:DRiceq1}\begin{aligned}
 -2D(\mathrm{Ric}_{g_\epsilon} + n) h_{j0} &= \epsilon^2 e^{-2t}
 \Delta_M h_{j0} - \epsilon^2 e^{-2t}
 g_M^{ik} \nabla^M_j\nabla^M_k h_{i0} \\
&\quad+
\epsilon^2 e^{-2t}\partial_0 \partial_j(g_M^{ik}h_{ik})-2\epsilon^2 e^{-2t}\partial_j(g_M^{ik}h_{ik})
\\
&\quad -\epsilon^2 e^{-2t} \partial_0 g_M^{ik}\nabla^M_k h_{ji} + 2
\epsilon^2 e^{-2t}
g_M^{ik}\nabla_k h_{ji} \\
&\quad -(n-1)\partial_j h_{00},
\end{aligned} \end{equation}
where $\Delta_M$ denotes the Hodge Laplacian on 1-form on $(M,g_M)$.

Finally, for the Bianchi operator we have
\begin{equation}\label{eq:bhformula}
 \begin{aligned} \beta_{g_\epsilon}(h)_j &= \epsilon^2e^{-2t}g_M^{ik} \nabla^M_k h_{ij} +
  nh_{j0} + \partial_0 h_{j0} - \frac{1}{2} \epsilon^2e^{-2t}\partial_j(g_M^{ik}h_{ik}) -
\frac{1}{2}\partial_j h_{00} \\
  \beta_{g_\epsilon}(h)_0 &= \epsilon^2e^{-2t} g_M^{ik} \nabla^M_k h_{i0} +
  \frac{1}{2} \partial_0 h_{00} + nh_{00} - \frac{1}{2} \epsilon^2 e^{-2t} \partial_0 (g_M^{ik} h_{ik}).
\end{aligned} \end{equation}

\subsection{Preliminary results on fixing the boundary values}
The results in this section will allow us to make sure that our
solutions of the linearized problem satisfy the Bianchi condition on
the boundary.

\begin{prop}\label{prop:fixboundary}
  Suppose that $\eta$ is a section of  $T^*(M\times
  [0,\infty))|_{t=0}$, in $C^{1,\alpha}$. We can find a symmetric 2-tensor
  $h\in C^{2,\alpha}$ on $M\times [0,\infty)$, supported in $M\times
  [0,1]$, such that
  \[ \beta_{g_\epsilon}(h)|_{t=0} = \eta, \]
  and in addition $\Vert h\Vert_{C^{2,\alpha}_1} \leq C\Vert
  \eta\Vert_{C^{1,\alpha}_{\epsilon^{-2}g_M}}$ for a constant $C$ independent of
  $\epsilon$, once $\epsilon$ is sufficiently small.
\end{prop}
\begin{proof}   The form $\eta$ decomposes as $\eta = \eta_i dx^i + \eta_0 dt$, where
  $\eta_i dx^i$ is a one form on $M$ and $\eta_0$ is a function on
  $M$. Using Lemma~\ref{lem:laplaceinverse} below, we can find a
  one-form $\omega_i$ and a function $f$ on $M$ such that
  \[ \begin{aligned}
      \Delta_{\epsilon^{-2} g_M} \omega_j - \omega_j &= \eta_j \\ 
      \Delta_{\epsilon^{-2} g_M} f - f &= \eta_0, 
  \end{aligned} \]
  where $\Delta_{\epsilon^{-2} g_M} = \epsilon^2 g_M^{ik} \nabla^M_i\nabla^M_k$ denotes the rough laplacian.  In addition, we have the estimates
  \begin{align} \label{omsize}
  \Vert  \omega\Vert_{C^{3,\alpha}_{\epsilon^{-2}g_M}} + \Vert f \Vert_{C^{3,\alpha}_{\epsilon^{-2}g_M}} &\leq C\Vert
  \eta\Vert_{C^{1,\alpha}_{\epsilon^{-2}g_M}}.
   \end{align}
  Define the 2-tensor $\tilde{h}$ on $M \times [0,\infty)$ by setting
  \[ \begin{aligned}
      \tilde{h}_{ij} &= \nabla^M_i \omega_j + \nabla^M_j \omega_i, \\
      \tilde{h}_{i0} &= \nabla^M_i f, \\
      \tilde{h}_{00} &= 0.
  \end{aligned} \]

  We have $\Vert \tilde{h}\Vert_{C^{2,\alpha}_1} \leq C \Vert
  \eta\Vert_{C^{1,\alpha}_{\epsilon^{-2}g_M}}$, and applying the Bianchi operator of
  $g_\epsilon$ according to \eqref{eq:bhformula},
  \[ \begin{aligned}
    \beta_{g_\epsilon}(\tilde{h})_j|_{t=0} &= \epsilon^2
    g_M^{ik}\nabla^M_k\tilde{h}_{ij} + n\tilde{h}_{j0} + \partial_0 \tilde{h}_{j0} -
    \frac{1}{2} \epsilon^2 g_M^{ik} \nabla^M_j \tilde{h}_{ik} \\
  &= \epsilon^2 g_M^{ik}\nabla^M_k(\nabla^M_i \omega_j +
  \nabla^M_j\omega_i) -\epsilon^2 g_M^{ik}\nabla^M_j\nabla^M_k\omega_i
  + n\nabla^M_j f \\
  &= \eta_j + \omega_j + \epsilon^2 g_M^{ik}\mathrm{Ric}^M_{kj}
  \omega_i + n \nabla^M_j f \\
   \beta_{g_\epsilon}(\tilde{h})_0|_{t=0} &= \epsilon^2 g_M^{ik}
   \nabla^M_k \tilde{h}_{i0} \\
   &= \epsilon^2 g_M^{ik}\nabla^M_k \nabla^M_i f \\
   &= \eta_0 + f.
  \end{aligned} \]
  It follows that if we write $\tau = \beta_{g_\epsilon}(\tilde{h})|_{t=0} - \eta$,
  then by (\ref{omsize})
  \[ \Vert \tau\Vert_{C^{2,\alpha}_{\epsilon^{-2}g_M}} \leq C
  \Vert\eta\Vert_{C^{1,\alpha}_{\epsilon^{-2}}}. \]
  We now define the 2-tensor $k$ on $M\times[0,\infty)$ by
  \[ \begin{aligned}
      k_{j0} &= t\tau_j \\
      k_{00} &= 2t\tau_0.
  \end{aligned} \]
  We have $\Vert k\Vert_{C^{2,\alpha}_1}\leq C\Vert
  \eta\Vert_{C^{1,\alpha}_{\epsilon^{-2}}}$ and in addition
  $\beta_{g_\epsilon}(k)|_{t=0} = \tau$  (note that $\tau$ vanishes
  when $t=0$, and so only the terms involving a $t$-derivative
  survive).

  Finally we define
  \[ h = \chi\cdot(\tilde{h} - k), \]
  where $\chi=\chi(t)$ is a cutoff function such that $\chi(t)=1$ for
  $t < 1/2$, and $\chi(t)=0$ for $t > 1$. Then $h$ is supported in
  $M\times[0,1]$, it satisfies the required $C^{2,\alpha}$ estimate
  since $\tilde{h}$ and $k$ do, and by construction it satisfies
  $\beta_{g_\epsilon}(h)|_{t=0} = \eta$.

\end{proof}

We have used the following result in the previous argument.
\begin{lem}\label{lem:laplaceinverse}
  Let $(M,g_M)$ be compact.
  For sufficiently small $\epsilon > 0$, and any $i$, the linear
  map $D: C^{3,\alpha}_{\epsilon^{-2} g_M}(\Omega^i (M)) \to C^{1,\alpha}_{\epsilon^{-2} g_{M}}(\Omega^i (M))$ given by
  \begin{align*}
  D:  \alpha \mapsto \Delta_{\epsilon^{-2} g_M } \alpha - \alpha,
  \end{align*}
  where $\Delta_{\epsilon^{-2} g_M} =  \epsilon^2 g_M^{jk} \nabla_j \nabla_k$ is the rough laplacian, is invertible.  Moreover, the inverse is bounded independently of $\epsilon$.
\end{lem}
\begin{proof}
  We will write down an approximate inverse for $D$. We cover $M$ with
  unit balls with respect to the metric $\epsilon^{-2}g_M$, and let
  $\gamma_1,\ldots, \gamma_{N_\epsilon}$ be a partition of unity subordinate to
  this cover. We have $N_\epsilon = O(\epsilon^{-n})$, and we can
  assume that all derivatives of the $\gamma_i$ are uniformly
  bounded. Given $u\in C^{1,\alpha}_{\epsilon^{-2}g_M}(\Omega^i(M))$,
  we write
  \[ u = \sum_{i=1}^{N_\epsilon} \gamma_i u. \]
  Using normal coordinates in each ball, we view $\gamma_i u$ as a
  tensor on $\mathbf{R}^n$ supported in the unit ball.
  On $\mathbf{R}^n$ we can solve the equation
  \[ \Delta_0 h_i - h_i = \gamma_i u, \]
  where $\Delta_0$ denotes the Euclidean Laplacian. Moreover, the
  solution decays in $C^\infty$ faster than any polynomial: for $|x| >
  1$ and any $k,d$, we have
  \[ |\partial_x^k h_i(x)| \leq C_{k,d} |x|^{-d} \Vert
  u\Vert_{C^{1,\alpha}_{\epsilon^{-2}g_M}},\]
  since the Green's function of the operator $\Delta_0 - 1$ on
  $\mathbf{R}^n$ decays exponentially fast (as can be seen using the
  Fourier transform for instance).

  We can now reassemble these local solutions $h_i$ as follows. We fix
  a radius $R > 2$, and let $\chi_R$ denote a cutoff function
  supported in $B_R(0)$, and equal to 1 in $B_{R-1}(0)$. By the
  decay of $h_i$ we have
  \begin{equation} \label{eq:chiRdelta} \Vert \Delta_0(h_i - \chi_R h_i)
  \Vert_{C^{1,\alpha}_{\epsilon^{-2}g_M}} \leq C_d R^{-d} \Vert
  u\Vert_{C^{1,\alpha}_{\epsilon^{-2}g_M}} \end{equation}
  for any $d > 0$.

  Once
  $\epsilon$ is sufficiently small, we can use normal coordinates to
  view each $\chi_R h_i$ as a tensor on $M$, supported in an
  $R$-ball. On such an $R$-ball, if we compare $\epsilon^{-2}g_M$ with
  the Euclidean metric $\delta_{ij}$ in normal coordinates, we have
  \[ \Vert \epsilon^{-2}g_{M, ij} - \delta_{ij} \Vert_{C^k} =
  O(\epsilon^2 R^2), \]
  and so
  \[ \Vert (\Delta_{\epsilon^{-2}g_M} - \Delta_0) (\chi_R
  h_i)\Vert_{C^{1,\alpha}_{\epsilon^{-2}g_M}} \leq C\epsilon^2 R^2
    \Vert u\Vert_{C^{1,\alpha}_{\epsilon^{-2}g_M}}. \]
  Combining this with \eqref{eq:chiRdelta} we obtain
  \[ \Vert \Delta_{\epsilon^{-2}g_M} (\chi_R h_i) - \chi_R h_i -
  \gamma_i u\Vert_{C^{1,\alpha}_{\epsilon^{-2}g_M}} \leq
  C_d(\epsilon^2R^2 + R^{-d}) \Vert
  u\Vert_{C^{1,\alpha}_{\epsilon^{-2}g_M}}. \]
  We now define
  \[ h = \sum_{i=1}^{N_\epsilon} \chi_R h_i, \]
  and estimate the error
 \begin{equation}\label{eq:errorest1}  \begin{aligned}
\Vert \Delta_{\epsilon^{-2}g_M} h - h - u
\Vert_{C^{1,\alpha}_{\epsilon^{-2}g_M}} &\leq R^n C_d(\epsilon^2R^2 +
R^{-d}) \Vert u\Vert_{C^{1,\alpha}_{\epsilon^{-2}g_M}}.
\end{aligned} \end{equation}
  In this estimate we used the fact that each $\chi_R h_i$ is
  supported on an $R$-ball, and so at each point of $M$, the number of
  terms that contribute is of order $R^n$. It is now clear that if we
  choose $d > n$, then $R$ sufficiently large, and finally $\epsilon$
  sufficiently small, we can ensure that
  \[
\Vert \Delta_{\epsilon^{-2}g_M} h - h - u
\Vert_{C^{1,\alpha}_{\epsilon^{-2}g_M}} \leq \frac{1}{2}\Vert u\Vert_{C^{1,\alpha}_{\epsilon^{-2}g_M}},
  \]
  and so the map $F: u \mapsto h$ that we defined is an approximate
  inverse for the linear operator $D$. In particular $DF$ is
  invertible, and $F(DF)^{-1}$ is the required inverse for $D$.
\end{proof}

We will also need the following, which allows us to correct the
boundary values when they only contain a $dt$ component.

\begin{prop}\label{prop:Binverse}
  Let $f \in C^{1,\alpha}(M,\mathbf{R})$. We can find $h\in
  C^{2,\alpha}_1(S^2)$ satisfying
  \begin{enumerate}
    \item $\Vert h\Vert_{C^{2,\alpha}_1} \leq C\Vert
      f\Vert_{C^{1,\alpha}_{\epsilon^{-2}g_M}}$,
    \item We have the boundary condition
      \[ \begin{aligned} \beta_{g_\epsilon}(h)_j|_{t=0} &= 0, \text{
          for all } j\\
        \beta_{g_\epsilon}(h)_0|_{t=0} &= f
        \end{aligned} \]
        \item $h$ is approximately in the kernel of $L_{g_{\epsilon}',g_{\epsilon}}$,
          in the sense that for any $\delta > 0$ there is a $C_\delta$
          such that
          \begin{equation}\label{eq:Lhest2} \Vert L_{g_{\epsilon}',g_{\epsilon}} h\Vert_{C^{0,\alpha}_1} \leq
          C_\delta\epsilon^{2-\delta}\Vert
          f\Vert_{C^{1,\alpha}_{\epsilon^{-2}g_M}}.
          \end{equation}
  \end{enumerate}
\end{prop}
\begin{proof}
 This result
should be compared with Proposition~\ref{prop:fixboundary}, where
arbitrary boundary values are allowed, but this comes at the cost of a
worse estimate for $L_{g_\epsilon',g_{\epsilon}}h$. The proof is
similar to the preceding proof, using the local result
Proposition~\ref{prop:Hinverse3}. As in the previous proof we write
\[ f = \sum_{i=1}^{N_\epsilon} \gamma_i f, \]
and apply Proposition~\ref{prop:Hinverse3} to each $\gamma_i f$. We
obtain $h_i$ satisfying $L_0(h_i)=0$, and $\beta_0(h_i)_0|_{t=0}=f$,
emphasizing that we are using the Euclidean operators $L_0$ and
$\beta_0$ here. We define
\[ \tilde{h} = \sum_{i=1}^{N_\epsilon} \chi_{R_\epsilon} h_i \]
as above, but we now allow the radius $R_\epsilon$ to depend on
$\epsilon$. Let us write $b$ for the tensor given by $b_j = 0$, and
$b_0=f$ on the slice $\{t=0\}$. Estimating the errors as in
\eqref{eq:errorest1},
we will have
\[  \Vert \beta_{g_\epsilon}(\tilde{h})|_{t=0} -
b\Vert_{C^{1,\alpha}_{g_\epsilon}} \leq R_\epsilon^n C_d( \epsilon^2
R_\epsilon^2 + R_\epsilon^{-d}) \Vert
f\Vert_{C^{1,\alpha}_{\epsilon^{-2}g_M}}. \]
Choosing $R_\epsilon = \epsilon^{-\tau}$ for some small $\tau > 0$,
this implies
\[ \begin{aligned} \Vert \beta_{g_\epsilon}(\tilde{h})|_{t=0} -
b\Vert_{C^{1,\alpha}_{g_\epsilon}} &\leq C_d( \epsilon^{2-(2+n)\tau} + \epsilon^{d\tau}) \Vert
f\Vert_{C^{1,\alpha}_{\epsilon^{-2}g_M}} \\
&\leq C_\delta \epsilon^{2-\delta} \Vert
f\Vert_{C^{1,\alpha}_{\epsilon^{-2}g_M}},
\end{aligned} \]
if $\tau = (2+n)^{-1}\delta$ and $d$ is sufficiently large. Similarly
we have
\[ \Vert L_{g_{\epsilon}', g_\epsilon} \tilde{h}\Vert_{C^{0,\alpha}_1} \leq
          C_\delta\epsilon^{2-\delta}\Vert
          f\Vert_{C^{1,\alpha}_{\epsilon^{-2}g_M}}. \]
We can now apply Proposition~\ref{prop:fixboundary} to perturb
$\tilde{h}$ to $h$ satisfying $\beta_{g_\epsilon}(h)|_{t=0} = b$,
while still satisfying the required estimate \eqref{eq:Lhest2} for
$L_{g_{\epsilon}',g_{\epsilon}}h$.
\end{proof}

\subsection{Inverting the linearized operator on the kernel of
  $I$}\label{sec:invertingLkerI}
We now move on to inverting the linearized operator, on the kernel of $I$.
Recall that given any symmetric 2-tensor $u$ on $M\times [0,\infty)$,
we defined the 1-form $I(u)$ on $M$ as in (\ref{Idef}). We then have the following.

\begin{prop}\label{prop:inverse1}
  We have a linear map
  \[ P_\epsilon : C^{0,\alpha}_1(\mathrm{ker}\,I) \to
  C^{2,\alpha}_1(S^2), \]
  satisfying the following.
  \begin{enumerate}
    \item There is a uniform bound $\Vert P_\epsilon
      u\Vert_{C^{2,\alpha}_1} \leq C \Vert u\Vert_{C^{0,\alpha}_1}$,
    \item $P_\epsilon$ is an approximate inverse to $L_{g_\epsilon}$
      in the sense that
       \[\Vert L_{g_{\epsilon}',g_{\epsilon}} P_\epsilon u -
         u\Vert_{C^{0,\alpha}_1} \leq C\epsilon^{p} \Vert
         u\Vert_{C^{0,\alpha}_1}, \]
       where we can take any $1 < p < 4/3$.
   \item $P_\epsilon(u)$ satisfies the Bianchi boundary condition,
     i.e. $\beta_{g_\epsilon}(P_\epsilon u)|_{t=0} = 0$.
 \end{enumerate}
 \end{prop}
\begin{proof}
  We construct an approximate inverse $\tilde{P}_\epsilon$ for
  $L_{g_{\epsilon}',g_{\epsilon}}$ in a very similar way to the proofs of
  Lemma~\ref{lem:laplaceinverse} and
  Proposition~\ref{prop:Binverse}. The difference is that the local
  result used here, Theorem~\ref{thm:Hinverse}, does not give
  rise to solutions with decay properties as strong as the local
  results used above. As a result the estimates required to obtain an
  approximate inverse are more delicate.

  As before, let us cover $M$ with unit balls $\{ B^i \}_{1 \leq i \leq N_{\epsilon}}$ with
  respect to the metric $\epsilon^{-2}g_M$, and let $\gamma_1,\ldots,
  \gamma_{N_\epsilon}$ be a partition of unity subordinate to this
  cover. We have $N_\epsilon = O(\epsilon^{-n})$, and we can assume
  that all derivatives of the $\gamma_i$ are bounded uniformly. Given
  $u\in C^{0,\alpha}_1(\mathrm{ker}\,I)$, we can express $u$ as
  \[ u = \sum_{i=1}^{N_\epsilon} \gamma_i u. \]

   Let $B^i = B^i_{\epsilon^{-2}g_M}(p_i,1)$ be one such ball of the covering, and let $\{ x^{\mu}\}$ be coordinates centered at $p_i$ that are normal
   with respect to $\epsilon^{-2}g_M$.  We may assume these coordinates are defined on all of $B^i$.  In particular, if $\{ y^{\mu} \}$ are coordinates centered at $p_i$ that are
   normal with respect to $g_M$, then we can just take $x^{\mu} = \epsilon^{-1} y^{\mu}$ to be the dilated coordinates, and it is clear that the $x$-coordinates are defined on $B^i$ once
   $\epsilon > 0$ is small enough.

   We can use the $x$-coordinates on $B^i$ to view $\gamma_i u$ as a 2-tensor on $\mathbf{H}^{n+1}_+$, supported in
  $B_1\times [0,\infty) \subset \mathbf{H}^{n+1}_+$, where $B_1$ is a (Euclidean) unit ball.  Also, $\gamma_i u$ satisfies
  $I(\gamma_i u)=0$, so we can apply Theorem \ref{thm:Hinverse}.  Letting $P_{\mathbf{H}}$ denote the inverse of the
  linearized operator $L_{g_{\mathbf{H}}}$ on the model
  space, we obtain solutions $h_i = P_{\mathbf{H}}(\gamma_i u)$ of $L_{g_{\mathbf{H}}}(h_i) =
  \gamma_i u$.  By the estimates of Theorem \ref{thm:Hinverse}, we have
  \[ \Vert h_i\Vert_{C^{2,\alpha}_1(A_{R-1,R})} \leq C \Vert
  u\Vert_{C^{0,\alpha}_1} R^{-n-1 + \delta}, \]
   where $\delta \in (0,1)$ and $R > 1$.  Let $R_\epsilon= \epsilon^{-2/3}$, and let $\chi_{R_\epsilon}$ be a cutoff function supported in the Euclidean ball
  $B_{R_\epsilon}$, equal to 1 in $B_{R_\epsilon-1}$.  Using the $x$-coordinates we can identify the balls $B_{\epsilon^{-2}g_M}(p_i,R_{\epsilon})$ with Euclidean
  balls $B(\mathbf{0},R_{\epsilon})$, and view $\chi_{R_\epsilon}h_i$ as a 2-tensor
  on $M\times [0,\infty)$.  We define
  \[ \tilde{P}_\epsilon(u) = \sum_{i=1}^{N_\epsilon} \chi_{R_\epsilon}h_i. \]

  In order to estimate the norm $\Vert \tilde{P}_\epsilon
  u\Vert_{C^{2,\alpha}_1}$, note that at each point $(p, t)\in M\times
  [0,\infty)$, there will be contributions to $\tilde{P}_\epsilon u$ from
  those $h_i$, for which the center of the corresponding ball in our
  covering of $M$ is of distance $k < R_\epsilon+1$ from $p$. There
  will be approximately $k^{n-1}$ balls whose distance from $p$ is in
  the interval $[k-1,k)$, and the corresponding functions
  $\chi_{R_\epsilon} h_i$ will contribute $k^{-n-1 + \delta}\Vert
  u\Vert_{C^{0,\alpha}_1}$ to the norm of $\tilde{P}_\epsilon u$ at $p$,
  because of the decay of $h_i$. Adding up these contributions we have
  \[ \begin{aligned}
      \Vert \tilde{P}_\epsilon u\Vert_{C^{2,\alpha}_1} &\leq C
      \sum_{k=1}^{\lceil R_\epsilon \rceil} k^{n-1} k^{-n-1 + \delta} \Vert
      u\Vert_{C^{0,\alpha}_1} \\
  &\leq C \Vert u\Vert_{C^{0,\alpha}_1},
\end{aligned} \]
since $\delta < 1$.   This gives the required bound on $\tilde{P}_\epsilon$.

  Next we need to estimate the error
  \[ \Vert L_{g_{\epsilon}',g_\epsilon} \tilde{P}_\epsilon(u) -
  u\Vert_{C^{0,\alpha}_1}. \]
  There are two sources of error: the difference between
  $L_{g_{\epsilon}', g_\epsilon}$ and $L_{g_{\mathbf{H}}}$, and the error from using
  the cutoff function $\chi_R$.

  For the latter note that
  \[ L_{g_{\epsilon}',g_\epsilon} (\chi_{R_\epsilon} h_i) = \chi_{R_\epsilon} L_{g_\epsilon}
  h_i + E = \gamma_i u_i + E, \]
  where $E$ is supported in $B_{R_\epsilon}\setminus
  B_{R_\epsilon-1}$, and it is bounded by the $C^{2,\alpha}$ norm of
  $h_i$ there. From the decay of $h_i$ we then get
  \begin{equation}\label{eq:1}
 \begin{aligned} \Vert L_{g_{\epsilon}',g_\epsilon}(\chi_{R_\epsilon} h_i) - \chi_{R_\epsilon}
  L_{g_{\epsilon}',g_\epsilon} h_i \Vert_{C^{0,\alpha}_1(A_{R_\epsilon-1, R_\epsilon})} &\leq C \Vert
  h_i\Vert_{C^{2,\alpha}_1(B_{R_\epsilon}\setminus B_{R_\epsilon-1})} \\
  &\leq C R_\epsilon^{-n-1 + \delta} \Vert u\Vert_{C^{0,\alpha}_1},
  \end{aligned}
  \end{equation}
  and the error vanishes outside the annular region $A_{R_\epsilon-1,
    R_\epsilon}$.

  Next we consider the error arising from the difference between $L_{g_\epsilon', g_{\epsilon}}$
  and $L_{g_{\mathbf{H}}}$.  To do this we first observe that on the set $C_{R} = B_{R} \times [0,\infty)$, where $1 < R < R_{\epsilon}$, we have
  \begin{equation}\label{eq:gM-gEuc}
 \Vert g_{\epsilon}' - g_{\mathbf{H}} \Vert_{C^2(C_{R})} = O(\epsilon^2 R^2).
\end{equation}
  Since $g_{\epsilon}'$ and $g_{\mathbf{H}}$ are close in $C^2$, we want to show the the corresponding linear operators are close (in a sense that will be made
  precise below).

  Recall the formula for $L_{g_{\epsilon}',g_{\epsilon}}$ in (\ref{Lform}):
  \begin{align*}
L_{g_{\epsilon}',g_{\epsilon}}(h) = -\frac{1}{2} \Delta_{g_{\epsilon}'} h + \mathfrak{D}_{g_{\epsilon}'}(h) + nh + \mathcal{R}_{g_{\epsilon}'}(h),
\end{align*}
where
\begin{align*}
\mathfrak{D}_{g_{\epsilon}'}(h) = \delta^*_{g_{\epsilon}'} \big\{  \beta_{g_{\epsilon}'}(h) -  \beta_{g_{\epsilon}}(h)  \big\},
\end{align*}
and $\mathcal{R}_{g_{\epsilon}'}$ is given by (\ref{Rterm}), with $\tilde{g} = g_{\epsilon}'$.  Also,
 \begin{align*}
L_{g_{\mathbf{H}}}(h) = -\frac{1}{2} \Delta_{g_{\mathbf{H}}} h + nh + \mathcal{R}_{g_{\mathbf{H}}}(h),
\end{align*}
where $\mathcal{R}_{g_{\mathbf{H}}}$ is now computed with respect to the curvature of $g_{\mathbf{H}}$.   Subtracting, we have
\begin{align} \label{diffs}
\big( L_{g_{\epsilon}',g_{\epsilon}} - L_{g_{\mathbf{H}}} \big)(h) = -\frac{1}{2} \big( \Delta_{g_{\epsilon}'} - \Delta_{g_{\mathbf{H}}} \big) h + \mathfrak{D}_{g_{\epsilon}'}(h)
+ \big( \mathcal{R}_{g_{\epsilon}'} - \mathcal{R}_{g_{\mathbf{H}}} \big)(h).
\end{align}
Therefore, we need to estimate each of the terms on the right-hand side.

To estimate the term with $\mathfrak{D}$, we use the fact that the Bianchi operator $\beta_g h$ with respect to a local coordinate system can be schematically written as
\begin{align*}
\beta_g h = g^{-1} * \partial h + g^{-1} * g^{-1} * \partial g * h,
\end{align*}
where $*$ denotes the operation of tensor products and contractions. It follows that given two metrics $g,g'$,
\begin{align} \label{bsc} \begin{split}
\beta_{g'} h - \beta_{g} h &= (  (g')^{-1} - g^{-1} ) * \partial h + ( (g')^{-1} - g^{-1}  ) * (g')^{-1} * \partial g' * h  \\
 & \ \ + \cdots + g^{-1} * g^{-1} * \partial ( g' - g) * h.
\end{split}
\end{align}
In the same way we can express the operator $\delta^{*}_g \omega$ as
\begin{align} \label{delsc}
\delta^{*}_g \omega = \partial \omega + g^{-1} * \partial g * \omega.
\end{align}
If we take $g = g_{\epsilon}$ and $g' = g_{\epsilon}'$, then combining (\ref{bsc}) and (\ref{delsc}) we have
\begin{align} \label{Dsc}
\mathfrak{D}_{g_{\epsilon}'}(h) =  (  (g_{\epsilon}')^{-1} - g_{\epsilon}^{-1} ) * \partial^2 h + \cdots  + g^{-1} * g^{-1} * \partial^2 ( g_{\epsilon}' - g_{\epsilon}) * h.
\end{align}
On the set $C_{R} = B_{R} \times [0,\infty)$, where $1 < R < R_{\epsilon}$, we have
  \begin{equation}\label{ggdiff}
 \Vert g_{\epsilon}' - g_{\epsilon} \Vert_{C^2(C_{R})} +  \Vert (g_{\epsilon}')^{-1} - g_{\epsilon}^{-1} \Vert_{C^2(C_{R})} = O(\epsilon^2 R^2).
\end{equation}
Consequently,
\begin{align} \label{Ddiff}
\Vert \mathfrak{D}_{g_{\epsilon}'}(h)  \Vert_{C^{0,\alpha}_1(A_{R-1, R})} \leq C  \epsilon^2 R^2 \Vert h \Vert_{C^{2,\alpha}_1(A_{R -1, R})},
\end{align}
 for $1 < R < R_{\epsilon}$.

To estimate the the remaining terms in (\ref{diffs}), we argue in a similar way.  For a metric $g$, we can schematically write the rough laplacian with respect to a local coordinate system as
 \begin{align*}
 \Delta_g h = g^{-1} * \partial^2 h + g^{-1}*\partial g * \partial h + g^{-1} * g^{-1} * \partial g * \partial g * h + g^{-1} * g^{-1} * \partial^2 g * h.
 \end{align*}
We can then estimate the difference $\big( \Delta_{g_{\epsilon}'} - \Delta_{g_{\mathbf{H}}} \big) h$ using (\ref{eq:gM-gEuc}) and (\ref{ggdiff}) to get
\begin{align*}
 \Vert (\Delta_{g_{\mathbf{H}}} - \Delta_{g_{\epsilon}'})(h) \Vert_{C^{0,\alpha}_1(A_{R-1, R})} \leq C  \epsilon^2 R^2 \Vert h \Vert_{C^{2,\alpha}_1(A_{R -1, R})}.
  \end{align*}
We can estimate the difference of the curvature terms in a similar manner, and combining all of these estimates we conclude
\begin{equation}\label{eq:LgHcompare}
 \Vert \big( L_{g_{\epsilon}',g_{\epsilon}} - L_{g_{\mathbf{H}}} \big)(h) \Vert_{C^{0,\alpha}_1(A_{R-1, R})} \leq C  \epsilon^2 R^2 \Vert h \Vert_{C^{2,\alpha}_1(A_{R -1, R})},
  \end{equation}
  for $1 < R < R_{\epsilon}$.
 Combining this with the errors introduced by the cut-off functions estimated in (\ref{eq:1}), we obtain
  \begin{equation}\label{eq:2} \begin{aligned}
        \Vert L_{g_{\epsilon}',g_\epsilon} (\chi_{R_\epsilon}h_i) -
        L_{g_{\mathbf{H}}} (\chi_{R_\epsilon}h_i)
        \Vert_{C^{0,\alpha}_1(A_{R-1,R})} &\leq C\epsilon^2 R^2 \Vert
        \chi_{R_\epsilon} h_i\Vert_{C^{2,\alpha}_1(A_{R-1,R})} \\
      &\leq C \epsilon^2 R^{-n + 1 + \delta } \Vert u\Vert_{C^{0,\alpha}_1},
  \end{aligned}
  \end{equation}
  for $1 < R < R_{\epsilon} + 1$, and the error vanishes for larger $R$.

  To estimate the difference $L_{g_{\epsilon}',g_\epsilon} P_\epsilon(u) - u$, we
  need to sum up all the contributions from \eqref{eq:1} and
  \eqref{eq:2}.
  \begin{enumerate}
\item For each $i$, the error coming from \eqref{eq:1} appears only in
  an annulus $B_{R_\epsilon} \setminus B_{R_\epsilon-1}$.
  Each point in $M$ will be covered by roughly
  $R_\epsilon^{n-1}$ such annuli, and so the total contribution of
  this type of error at each point will be bounded by
 \[CR_\epsilon^{n-1}R_\epsilon^{-n-1 + \delta } \Vert u\Vert_{C^{0,\alpha}_1} =
 CR_\epsilon^{-2 + \delta } \Vert u\Vert_{C^{0,\alpha}_1}. \]

\item For each $i$, the error coming from \eqref{eq:2} appears on an
  $R_\epsilon$-ball, but it decays as we approach the boundary of the
  $R_\epsilon$-ball. We estimate this in a similar way to the way we
  bounded $P_\epsilon u$ above. When $\epsilon$ is sufficiently small, then on an
  $R_\epsilon$ ball our cover of $(M, \epsilon^{-2} g_M)$ with
  unit balls has centers that are roughly on the grid $\mathbf{Z}^n\subset
  \mathbf{R}^n$ (in normal coordinates). We can sum up the
  contributions of these errors at the origin. If a unit ball
  has  distance in the interval $[k-1, k)$ from the origin, then
  according to \eqref{eq:2} it
  contributes an error of $C\epsilon^2  k^{-n + 1 + \delta}\Vert
  u\Vert_{C^{0,\alpha}_1}$.
  There will be roughly
  $k^{n-1}$ such balls, and $k$ can range from $1$ to $\lceil
  R_\epsilon\rceil$.
  The sum of
  errors will therefore be bounded by
  \[ C\epsilon^2  \sum_{k=1}^{\lceil R_\epsilon \rceil} k^{n-1}k^{- n + 1 + \delta }
  \Vert u\Vert_{C^{0,\alpha}_1}
  \leq C \epsilon^2 R_\epsilon^{1 + \delta} \Vert u\Vert_{C^{0,\alpha}_1}. \]
\end{enumerate}

Adding up all of these contributions we have
\[ \begin{aligned}
 \Vert L_{g_{\epsilon}',g_\epsilon}\tilde{P}_\epsilon(u) - u\Vert_{C^{0,\alpha}_1} &\leq
C(R_\epsilon^{-2 + \delta } + \epsilon^2 R_{\epsilon}^{1 + \delta} ) \Vert
u\Vert_{C^{0,\alpha}_1} \\
&\leq C\epsilon^{\frac{4}{3} - \frac{2}{3}\delta} \Vert
u\Vert_{C^{0,\alpha}_1},
\end{aligned} \]
since $R_\epsilon = \epsilon^{-2/3}$.  As $\delta \in (0,1)$, we conclude
\begin{align*}
 \Vert L_{g_{\epsilon}',g_\epsilon}\tilde{P}_\epsilon(u) - u\Vert_{C^{0,\alpha}_1} \leq C\epsilon^{p} \Vert
u\Vert_{C^{0,\alpha}_1},
\end{align*}
for any $1 < p < 4/3$.

    We still need to consider the boundary condition. Each $h_i = P(\gamma_i
    u)$ satisfies the boundary condition with respect to the
    hyperbolic metric, but we introduce an error when we multiply with
    the cutoff function $\chi_R$, and also when we use the metric
    $g_\epsilon$ instead of the hyperbolic metric. Accounting for the
    errors exactly as above, we have
   \[ \Vert \beta_{g_\epsilon}(\tilde{P}_\epsilon
   u)\Vert_{C^{1,\alpha}_1} \leq C\epsilon^{p} \Vert
   u\Vert_{C^{0,\alpha}_1}, \]
   where $1 < p < 4/3$.  We can now us Proposition~\ref{prop:fixboundary} to find a 2-tensor
   $k$ supported in $M\times [0,1]$, satisfying $\Vert
   k\Vert_{C^{2,\alpha}_1} \leq C\epsilon^{p} \Vert
   u\Vert_{C^{0,\alpha}_1}$ and   $\beta_{g_\epsilon}(k)|_{t=0} =
   \beta_{g_\epsilon}(\tilde{P}_\epsilon u)|_{t=0}$. Then we define
  \[ P_\epsilon u = \tilde{P}_\epsilon u - k, \]
   and this will satisfy all of our requirements.
\end{proof}

\subsection{The induced operator on $\Omega^1(M)$}
In the previous section we considered the equation
$L_{g_{\epsilon}',g_{\epsilon}}(h)=u$ for $u\in \mathrm{ker}\,I$. We now consider the
complementary problem of solving $I\circ L_{g_{\epsilon}',g_{\epsilon}}(h) = \omega$
for a one-form $\omega$ on $M$. Let us define the operator
\begin{align} \label{Tdef} \begin{split}
    \mathcal{T}: C^{2,\alpha}_{g_M}(\Omega^1(M)) &\to
    C^{0,\alpha}_{g_M}
    (\Omega^1(M)), \\
    \omega &\mapsto I\big(
    L_{g_{\epsilon}',g_{\epsilon}}(\epsilon^{-2}e^{-nt}\omega\odot dt)\big),
    \end{split}
\end{align}
where we emphasize that we will now measure norms using the metric
$g_M$ on $M$ instead of $\epsilon^{-2}g_M$.

The dependence of $\mathcal{T}$ on $\epsilon$ is described by the
following result.
\begin{prop}
  There is an elliptic operator
  \[ \mathcal{T}_0 : C^{2,\alpha}_{g_M}(\Omega^1(M)) \to
    C^{0,\alpha}_{g_M}(\Omega^1(M)) \]
  such that
  \begin{equation}\label{eq:TT_0}
 \Vert \mathcal{T} - \mathcal{T}_0 \Vert \leq C\epsilon^{1-\alpha},
  \end{equation}
  for a constant $C$ independent of $\epsilon$.
\end{prop}
\begin{proof}
To begin, we want to view $L_{g_\epsilon', g_\epsilon}$ as a perturbation of the hyperbolic model operator $L_{g_{\mathbf{H}}}$.
This will require us to use normal coordinates to identify the one-form $\omega$ on $M$ with a one-form
  in $\mathbf{H}$.   To this end, as in the proof of Proposition \ref{prop:inverse1} we let $\{ x^{\mu}\}$ denote normal coordinates with respect to $\epsilon^{-2}g_M$
  defined on a ball in $M$.  With respect to these coordinates, on the region $M\times [0,1]$ we have
  \[ \epsilon^{-2}g_M = \delta_{ij} + O(\epsilon^2), \]
  (see (\ref{eq:gM-gEuc})).  In addition, if we use these coordinates to define the hyperbolic metric $g_{\mathbf{H}} = dt^2 + e^{2t} dx^2$, then $k^{(2)}, k^{(4)} = O(\epsilon^2)$, so
  \begin{equation}\label{eq:gexpand}
    \begin{aligned}  g_\epsilon &= g_{\mathbf{H}} + O(\epsilon^2), \\
      g_\epsilon' &= g_{\mathbf{H}} + O(\epsilon^2).
      \end{aligned}
  \end{equation}

  Using the estimates in the proof of Proposition \ref{prop:inverse1}), we can write
  \[ L_{g_\epsilon', g_\epsilon} = L_{g_{\mathbf{H}}} + \epsilon^2
    \mathcal{P} + O(\epsilon^3), \]
  where $\mathcal{P}$ is a linear operator independent of $\epsilon$,
  determined by the terms of order $\epsilon^2$ in the difference (\ref{eq:LgHcompare}).   Note also that if $\Vert \omega\Vert_{C^{2,\alpha}_{g_M}} \leq 1$,
  then in the $x$-coordinates we have
  \begin{equation} \label{eq:omegaests}
   |\omega_j| \leq C\epsilon, \quad |\partial \omega_j| \leq
    C\epsilon^2, \quad |\partial^2\omega_j|_{C^\alpha} \leq
    C\epsilon^3.
  \end{equation}
  It follows that
  \[ L_{g_\epsilon', g_\epsilon} (\epsilon^{-2}e^{-nt}\odot dt) =
    L_{g_\mathbf{H}}(\epsilon^{-2}e^{-nt}\omega\odot dt) +
    \mathcal{P}(e^{-nt} \omega\odot dt) + O(\epsilon^2). \]
 For tensors of the form $h =  \epsilon^{-2}e^{-nt}\omega\odot dt$, it follows from (\ref{Lhomog}) that
  \[ L_{g_\mathbf{H}}(\epsilon^{-2}e^{-nt}\omega\odot dt)_{j0} =
    e^{-2t} \epsilon^{-2} \partial_i\partial_i \omega_j \]
Let us write
  \[ \mathcal{P}(e^{-nt}\omega\odot dt)_{j0} = A_2(\omega)_{j0} +
    A_1(\omega)_{j0} + A_0(\omega)_{j0}, \]
  where $A_m$ denotes the degree $m$ part of the operator. The
  estimates \eqref{eq:omegaests} imply that
  \[ L_{g_\epsilon', g_\epsilon} (\epsilon^{-2}e^{-nt}\omega\odot dt)_{j0} =
    e^{-2t}\epsilon^{-2} \partial_i\partial_i \omega_j +
    A_0(\omega)_{j0} + O(\epsilon^2), \]
    where again we emphasize the components are with respect to the $x$-coordinates.

  The same calculation also applies on the regions $M\times [T, T+1]$
  for all $T$, using normal coordinates for
  $e^{2T}\epsilon^{-2}g_M$. Applying the operator $I$
  (i.e. integrating out the $t$ variable),  we find that at least at
  the center of our coordinate system we have
  \begin{equation}\label{eq:ILerror}
    I(L_{g_\epsilon', g_\epsilon}(\epsilon^{-2}e^{-nt}\omega \odot
  dt))_j = c \epsilon^{-2}\partial_i\partial_i \omega_j +
  A(\omega)_j + O(\epsilon^2),
  \end{equation}
  where $c$ is a fixed constant arising from integrating the
  exponential term in $t$, and $A$ is a zeroth order operator on
  one-forms. Note that up to zeroth order terms, at the origin of our
  normal coordinate system
  $\partial_i\partial_i \omega_j$ is simply the rough
  Laplacian of $\omega$ with respect to the metric $g_M$. When we
  measure the $O(\epsilon^2)$ error term in \eqref{eq:ILerror} with
  respect to $g_M$ instead of $\epsilon^{-2}g_M$, then in the
  $C^{0,\alpha}$-norm we lose a factor of $\epsilon^{1+\alpha}$. In
  sum we have
  \[ \Vert \mathcal{T}\omega - c \Delta_{g_M} \omega - \tilde{A}
  \omega\Vert_{C^{0,\alpha}} \leq C\epsilon^{1-\alpha}, \]
  where $\tilde{A}$ is a zeroth order operator and $\Delta_{g_M}$ is
  the rough Laplacian on one-forms.
\end{proof}

The specific form of $\mathcal{T}_0$ is not important, but note that
for instance if instead of $g_\epsilon'$ we use the metric
$g_\epsilon$, then $\mathcal{T}_0$ is the Hodge Laplacian on one
forms. In particular $\mathcal{T}_0$ is not necessarily surjective,
already in this simple case. It is for this reason that we introduce a
further finite dimensional space $E\subset C^{2,\alpha}_1(S^2)$, and
consider the linear operator
\[ \begin{aligned}
 \overline{\mathcal{T}} : E \times
    C^{2,\alpha}_{g_M}(\Omega^1(M)) &\to C^{0,\alpha}_{g_M}(\Omega^1(M)) \\
    (r, \omega) &\mapsto I\Big[ (D\mathrm{Ric}_{g_\epsilon'} + n)
    r + L_{g_{\epsilon}',g_{\epsilon}}(\epsilon^{-2}e^{-nt}
    \omega\odot dt)\Big].
\end{aligned} \]
We then have the following.
\begin{prop}\label{prop:Tbarinverse}
For a suitable finite dimensional subspace $E\subset C^{2,\alpha}_1$,
the operator $\overline{\mathcal{T}}$ has a
right inverse with bound independent of $\epsilon$, as long as
$\epsilon$ is sufficiently small. Here the norm on
$E$ is the $C^{2,\alpha}_1$ norm.
\end{prop}
\begin{proof}
Suppose that $r$ is a tensor of the form
\[ r = \epsilon^{-2}e^{-nt}h + f e^{-nt} dt\odot dt, \]
where $h$ is a symmetric two-tensor on $M$, and $f$ is a function on
$M$. Suppose in addition that $\mathrm{tr}_{g_M}h = 0$. We then have
the following formula (see \eqref{eq:DRiceq1}):
\[ -2D(\mathrm{Ric}_{g_\epsilon} + n)r_{j0} = (n+2)e^{-(n+2)t}
  g_M^{ik}\nabla^M_k h_{ij} - (n-1) e^{-nt} \partial_j f. \]
It follows that for some nonzero dimensional constants $c_1, c_2$ we
have
\[ I\circ D(\mathrm{Ric}_{g_\epsilon}+n)r = c_1 \delta_{g_M} h + c_2
  df. \]
A calculation shows that replacing $g_\epsilon$ by $g_\epsilon'$ only
introduces lower order terms. More precisely
\begin{equation}
\label{eq:IDerror} \Vert I\circ D(\mathrm{Ric}_{g_\epsilon'} + n)r - c_1\delta_{g_M} h
  - c_2 df\Vert_{C^{0,\alpha}_{g_M}} \leq C\epsilon^2 (\Vert
  h\Vert_{C^{2,\alpha}_{g_M}} + \Vert f\Vert_{C^{2,\alpha}_{g_M}}).
\end{equation}
We now observe that on the space of symmetric 2-tensors on $M$, the
operator $\delta_{g_M}$ is underdetermined elliptic, and so its image
is the orthogonal complement of $\mathrm{ker}\, \delta_{g_M}^*$, which
can be identified with the space of Killing vector fields. A generic
metric in the conformal class of $g_M$ has no Killing fields, and so
$\delta_{g_M}$ is surjective. We assume from now that this is the
case. Given any one-form $\eta$ on $M$, we can
then find $h\in S^2(M)$ such that $\delta h = \eta$, and so
\[ \eta = \delta\left\{ h - \frac{1}{n} (\mathrm{tr}_{g_M}h)g_M\right\}
  + \frac{1}{n} d\mathrm{tr}_{g_M}h. \]
It follows that for any $\eta \in \mathrm{coker}\, \mathcal{T}_0$ we
can find $r$ as above, such that
\[ I\circ D(\mathrm{Ric}_{g_\epsilon} + n)r = \eta. \]
Moreover since $\mathrm{coker}\,\mathcal{T}_0$ is finite dimensional, we have
\[ \Vert h\Vert_{C^{2,\alpha}_{g_M}} + \Vert
  f\Vert_{C^{2,\alpha}_{g_M}} \leq C\Vert
  \eta\Vert_{C^{0,\alpha}_{g_M}}. \]
It follows that
\[ \Vert r\Vert_{C^{2,\alpha}_1} \leq C\Vert
  \eta\Vert_{C^{0,\alpha}_{g_M}}. \]
We can then use this, together with \eqref{eq:TT_0} and
\eqref{eq:IDerror} to show the invertibility of
$\overline{\mathcal{T}}$ for sufficiently small $\epsilon$.
\end{proof}

\subsection{Inverting the full linearized operator}\label{sec:invertingLbar}
We now combine the pieces developed in the previous sections. We
consider the linearized operator
\[ \begin{aligned}
    \overline{L}_{g_\epsilon'} : E\times (C^{2,\alpha}_1)_\beta &\to
    C^{0,\alpha}_1 \\
    (r,h) &\mapsto (D\mathrm{Ric}_{g_\epsilon'} + n)r +
    L_{g_{\epsilon}',g_{\epsilon}}(h),
\end{aligned} \]
where $E$ is a finite dimensional subspace of $C^{2,\alpha}_1$ as above.
We can now prove Theorem~\ref{thm:Lbarinverse} on finding a right
inverse for $\overline{L}_{g_\epsilon'}$. We state the result here
again.
\begin{thm}
  Suppose that $(M,g_M)$ admits no Killing vector fields. Then for
  sufficiently small $\epsilon$ and $\alpha$ the operator
  $\overline{L}_{g_\epsilon'}$ has a right inverse $\mathcal{R}$,
  satisfying $\Vert \mathcal{R}\Vert \leq C\epsilon^{-2-\alpha}$.
\end{thm}
\begin{proof}
  We construct an approximate inverse. Let $u\in C^{0,\alpha}_1$, with
  $\Vert u\Vert_{C^{0,\alpha}_1}\leq 1$. Then
  $I(u)$ is a one-form on $M$ satisfying the estimate
  \[ \Vert I(u)\Vert_{C^{0,\alpha}_{g_M}} \leq C \epsilon^{-1-\alpha}. \]
  From Proposition~\ref{prop:Tbarinverse} we have $r\in
  E$ and $\omega\in C^{2,\alpha}_{g_M}(\Omega^1M)$
  such that
  \[  \overline{\mathcal{T}}_{g_\epsilon'}(r,\omega) = I(u), \]
  and
  \[ \Vert r\Vert_{C^{2,\alpha}_1} \leq C\epsilon^{-1-\alpha}, \quad \Vert
    \omega\Vert_{C^{2,\alpha}_{g_M}} \leq C\epsilon^{-1-\alpha}. \]
  We let
  \[ u_1 = u - (D\mathrm{Ric}_{g_\epsilon'}+n)r -
    L_{g_{\epsilon}',g_{\epsilon}}(\epsilon^{-2}e^{-nt}\omega\odot dt), \]
  so that by construction, $u_1 \in \mathrm{ker}\,I$. At the same time
  \[ \Vert D(\mathrm{Ric}_{g_\epsilon'} +
    n)r\Vert_{C^{0,\alpha}_1}\leq C\Vert r\Vert_{C^{2,\alpha}_1} \leq
    C\epsilon^{-1-\alpha}, \]
  and we also have
  \begin{equation}\label{eq:Lomegaest1}
 \Vert L_{g_{\epsilon}',g_{\epsilon}}(\epsilon^{-2}e^{-nt}\omega\odot dt)
    \Vert_{C^{0,\alpha}_1} \leq C\epsilon^{-1-\alpha}.
  \end{equation}
  For the latter estimate note that by the formulas \eqref{Lhomog}, in
  the model hyperbolic space, for a tensor of the form $v =
  e^{-nt}\omega\odot dt$ with an $n$-form $\omega$ on $\mathbf{R}^n$,
  we have
  \[ \Vert L_{g_{\mathbf{H}}} v\Vert_{C^{0,\alpha}_1} \leq C\Vert
    \nabla\omega\Vert_{C^{1,\alpha}_{\mathbf{R}^n}}, \]
  since the terms that involve only $t$-derivatives of $h$
  cancel. Arguing similarly to \eqref{eq:LgHcompare} we then find that
  \[ \begin{aligned} \Vert L_{g_{\epsilon}',g_{\epsilon}} v\Vert_{C^{0,\alpha}_1} &\leq C\Big[\Vert
    \nabla\omega\Vert_{C^{1,\alpha}_{\epsilon^{-2}g_M}} +
    \epsilon^2\Vert
    \omega\Vert_{C^{2,\alpha}_{\epsilon^{-2}g_M}}\Big] \\
    &\leq C\epsilon^2\Vert \omega\Vert_{C^{2,\alpha}_{g_M}}.
\end{aligned}
\]
  The estimate \eqref{eq:Lomegaest1} then follows, and as a
  consequence we have
  \[ \Vert u_1\Vert_{C^{0,\alpha}_1} \leq C\epsilon^{-1-\alpha}. \]

  We now invoke Proposition~\ref{prop:inverse1}, to find $h_1 =
  P_\epsilon u_1$, satisfying
  \begin{equation}\label{eq:h1props} \begin{aligned}
          \Vert h_1\Vert_{C^{2,\alpha}_1} &\leq C\epsilon^{-1-\alpha},
          \\
          \Vert L_{g_{\epsilon}',g_{\epsilon}}h_1 - u_1\Vert_{C^{0,\alpha}_1} &\leq
          C\epsilon^{p-1-\alpha}, \\
          \beta_{g_\epsilon}(h_1)|_{t=0} &=0,
\end{aligned} \end{equation}
where note that we can choose any $p \in (1,4/3)$.

  To construct our approximate solution what remains is to take care
  of the Bianchi boundary condition for the term
  $\epsilon^{-2}e^{-nt}\omega\odot dt$. Note that by
  \eqref{eq:bhformula} we have
  \[ \begin{aligned}
      \beta_{g_\epsilon}(\epsilon^{-2}e^{-nt}\omega\odot dt)_i|_{t=0} &= 0, \\
      \beta_{g_\epsilon}(\epsilon^{-2}e^{-nt}\omega\odot dt)_0|_{t=0}
      &= g_M^{ij}\nabla_i^M \omega_j.
      \end{aligned} \]
   We can apply Proposition~\ref{prop:Binverse} to find a two-tensor
   $k$ satisfying
   \[\beta_{g_\epsilon}(k)|_{t=0} =
   \beta_{g_\epsilon}(\epsilon^{-2}e^{-nt} \omega \odot dt)|_{t=0}, \]
   and the estimates
   \[ \begin{aligned}
     \Vert k\Vert_{C^{2,\alpha}_1}\leq C\Vert
     \omega\Vert_{C^{2,\alpha}_{g_M}} &\leq C\epsilon^{-1-\alpha}, \\
     \Vert L_{g_\epsilon', g_\epsilon} k\Vert_{C^{0,\alpha}_1} &\leq C_\delta
     \epsilon^{-1-\alpha}\epsilon^{2-\delta}\\
    &\leq C\epsilon^{1-\alpha-\delta},
\end{aligned} \]
  for any $\delta > 0$.

  We now set
  \[ h = h_1 + \epsilon^{-2}e^{-nt} \omega\odot dt - k. \]
  By construction we have
  \[ \beta(h)|_{t=0} = 0 \]
    as required, and
  \[ \Vert \omega\Vert_{C^{2,\alpha}_{\epsilon^{-2}g_M}} \leq
  C\epsilon \Vert \omega\Vert_{C^{2,\alpha}_{g_M}} \leq C
  \epsilon^{-\alpha} \]
  implies
  \[ \Vert h\Vert_{C^{2,\alpha}_1} \leq C\epsilon^{-2-\alpha}.\]
  In addition we have $\Vert r\Vert_{C^{2,\alpha}_1} \leq C
  \epsilon^{-1-\alpha}$, and
\[ \begin{aligned}
    &\Big\Vert u -
    \overline{L}_{g_\epsilon'}(r,h)  \Big\Vert_{C^{0,\alpha}_1} \\
    =\Big\Vert u - (D\mathrm{Ric}_{g_\epsilon'} + n) r -
    &L_{g_{\epsilon}',g_{\epsilon}}(\epsilon^{-2} e^{-nt} \omega\odot dt) -
    L_{g_{\epsilon}',g_{\epsilon}}h_1 + L_{g_{\epsilon}',g_{\epsilon}}k \Big\Vert_{C^{0,\alpha}_1}
\\ &= \Big\Vert u_1 - L_{g_{\epsilon}',g_{\epsilon}}h_1 + L_{g_{\epsilon}',g_{\epsilon}}k
\Vert_{C^{0,\alpha}_1} \\
&\leq C(\epsilon^{p-1-\alpha} + \epsilon^{1-\alpha-\delta}).
\end{aligned} \]
If $\alpha$ is sufficiently small, $p > 1$ and $\delta$ is small, then
for sufficiently small $\epsilon$ the map $u\mapsto (r,h)$ is then an
approximate inverse for $\overline{L}_{g_\epsilon'}$, and we can
perturb it to a genuine inverse.
\end{proof}

%
%
%

\end{document}